\documentclass[dvipsnames]{siamart190516}

\usepackage[utf8]{inputenc}
\usepackage{bbm,mathrsfs,mathtools,amsmath,amssymb,amsfonts,stmaryrd,upgreek}
\usepackage{fancyhdr,subfigure}
\usepackage{epsfig,graphicx,wrapfig,fancyvrb,listings,color,textcomp,float}
\usepackage{algorithmic}
\Crefname{ALC@unique}{Line}{Lines}
\usepackage{enumitem,url,hyperref}
\usepackage{sidecap,titlesec,scalerel}
\sidecaptionvpos{figure}{c} % center the text

\usepackage{tikz}
\usetikzlibrary{positioning}
\usetikzlibrary{calc}

\newcommand{\obs}{ {\mathbf{Y}} }
\newcommand{\Obs}{ {Y} }

\newif\ifcomplete

\newcommand{\Hobs}{\mathcal{H}}
\newcommand{\HH}{\mathbf H}

\newcommand{\K}{ \mathbf{K} }

\newcommand{\Z}{ \mathbf{Z} }

\usepackage{accents}
\newcommand*{\bigdot}[1]{\accentset{\mbox{\large\bfseries .}}{#1}}

\newcommand{\Model}{\mathcal{M}}

\newcommand{\x}[1][]{%
   \ifthenelse{ \equal{#1}{} }
      {X}
      {X^{[#1]}}}
      
\renewcommand{\u}{U}
\newcommand{\uh}[1][]{%
   \ifthenelse{ \equal{#1}{} }
      {\hat{U}}
      {\hat{u}^{(#1)}}}
\newcommand{\xb}[1][]{%
   \ifthenelse{ \equal{#1}{} }
      {X^{\mathrm{b}}}
      {\mathbf{X}^{\mathrm{b},[#1]}}}
      
\newcommand{\ub}[1][]{%
   \ifthenelse{ \equal{#1}{} }
      {U^{\mathrm{b}}}
      {\mathbf{U}^{\mathrm{b},[#1]}}}
      
\newcommand{\ua}[1][]{%
   \ifthenelse{ \equal{#1}{} }
      {U^{\mathrm{a}}}
      {\mathbf{U}^{\mathrm{a},[#1]}}}
      
\newcommand{\uha}[1][]{%
   \ifthenelse{ \equal{#1}{} }
      {\hat{U}^{\mathrm{a}}}
      {\mathbf{\hat{U}}^{\mathrm{a},[#1]}}}
      
\newcommand{\pub}[1][]{%
   \ifthenelse{ \equal{#1}{} }
      {\*\Phi_r\,U^{\mathrm{b}}}
      {\*\Phi_r\,\mathbf{U}^{\mathrm{b},[#1]}}}
      
\newcommand{\uhb}[1][]{%
   \ifthenelse{ \equal{#1}{} }
      {\hat{U}^{\mathrm{b}}}
      {\mathbf{\hat{U}}^{\mathrm{b},[#1]}}}
      
\newcommand{\puhb}[1][]{%
   \ifthenelse{ \equal{#1}{} }
      {\*\Phi_r\,\hat{U}^{\mathrm{b}}}
      {\*\Phi_r\,\mathbf{\hat{U}}^{\mathrm{b},#1}}}
      
\newcommand{\zb}[1][]{%
   \ifthenelse{ \equal{#1}{} }
      {Z^{\mathrm{b}}}
      {\mathbf{Z}^{\mathrm{b},#1}}}
      
\newcommand{\za}[1][]{%
   \ifthenelse{ \equal{#1}{} }
      {Z^{\mathrm{a}}}
      {\mathbf{Z}^{\mathrm{a},[#1]}}}

\newcommand{\xa}[1][]{%
   \ifthenelse{ \equal{#1}{} }
      {X^{\rm a}}
      {\mathbf{X}^{\mathrm{a},[#1]}}}
\newcommand{\xf}[1][]{%
   \ifthenelse{ \equal{#1}{} }
      {\mathbf{x}^{\rm f}}
      {\mathbf{x}^{{\rm f}[#1]}}}
\newcommand{\xt}{ \mathbf{X}^{\mathrm{t}} }

\newcommand{\hofx}{\mathbf{z}}

\newcommand{\hofxb}[1][]{%
   \ifthenelse{ \equal{#1}{} }
      {\hofx^{\mathrm{b}}}
      {\hofx^{{\mathrm{b}}[#1]}}}
\newcommand{\hofxa}[1][]{%
   \ifthenelse{ \equal{#1}{} }
      {\hofx^{\rm a}}
      {\hofx^{{\rm a}[#1]}}}
\newcommand{\dhofxb}[1][]{%
   \ifthenelse{ \equal{#1}{} }
      {\bigdot{\hofx}^{\mathrm{b}}}
      {\bigdot{\hofx}^{{\mathrm{b}}[#1]}}}

\newcommand{\one}{\mathbf{1}}
\newcommand{\errb}[1][]{%
   \ifthenelse{ \equal{#1}{} }
      {\varepsilon^{\mathrm{b}}}
      {\varepsilon^{{\mathrm{b}}[#1]}}}
\newcommand{\erra}[1][]{%
   \ifthenelse{ \equal{#1}{} }
      {\varepsilon^{\rm a}}
      {\varepsilon^{{\rm a}[#1]}}}
\newcommand{\erro}[1][]{%
   \ifthenelse{ \equal{#1}{} }
      {\eta}
      {\boldsymbol{\eta}^{(#1)}}}

\newcommand{\errm}[1][]{%
   \ifthenelse{ \equal{#1}{} }
      {\eta}
      {\eta^{[#1]}}}
\newcommand{\wb}[1][]{%
   \ifthenelse{ \equal{#1}{} }
      {w^{\mathrm{b}}}
      {w^{{\mathrm{b}}[#1]}}}
\newcommand{\wa}[1][]{%
   \ifthenelse{ \equal{#1}{} }
      {w^{\rm a}}
      {w^{{\rm a}[#1]}}}

\renewcommand{\Re}{\mathbbm{R}}

\newtheorem{remark}{Remark}
\newtheorem{assumption}{Assumption}

\newcommand{\norm}[1]{\bigl\Vert #1 \bigr\Vert}

\newcommand{\Nens}{{{\rm N}}}

\renewcommand{\thesubsubsection}{\thesubsection.\arabic{subsubsection}}
\makeatletter
\renewcommand{\p@subsection}{}
\renewcommand{\p@subsubsection}{}
\makeatother

\def\!#1{\mathcal{#1}}
\def\*#1{\boldsymbol{\mathbf{#1}}}
\def\|#1{\textnormal{#1}}

\def\norm#1{\left\lVert#1\right\rVert}

\newcommand{\Mean}[1]{\*\mu_{#1}}
\newcommand{\MeanE}[1]{\widetilde{\*\mu}_{#1}}
\newcommand{\Cov}[2]{\*\Sigma_{#1,#2}}
\newcommand{\CovE}[2]{\widetilde{\*\Sigma}_{#1,#2}}
\newcommand{\En}[1]{\mathsf{E}_{#1}}
\newcommand{\An}[1]{\mathsf{A}_{#1}}
\renewcommand{\Z}[2][\Hobs_i]{\mathsf{A}_{#1(#2)}}

\newcommand{\inner}[2]{\left\langle#1,#2\right\rangle}

\DeclareMathOperator{\tr}{tr}

\graphicspath{ {Figures/} }

\newsiamthm{conjecture}{Conjecture}%
\crefname{conjecture}{Conjecture}{Conjectures}

\def\realtitle{A Multifidelity Ensemble Kalman Filter with Reduced Order Control Variates}

\title{{\realtitle}\thanks{Submitted to the arXiv \today.
		\funding{The work of Popov and Sandu was supported by awards NSF CCF--1613905 and NSF ACI--1709727,
			and by the Computational Science Laboratory at Virginia Tech.
			The work of Mou and Iliescu was supported by awards NSF DMS--1821145 and NSF CMMI--1929731.}}}
\date{\today}
\author{Andrey A. Popov\thanks{Computational Science Laboratory, Department of Computer Science, Virginia Tech, Blacksburg, VA
  (\email{apopov@vt.edu}, \email{sandu@cs.vt.edu}).}
\and Changhong Mou\thanks{Department of Mathematics, Virginia Tech, Blacksburg, VA
  (\email{cmou@vt.edu}, \email{iliescu@vt.edu}).}
\and Traian Iliescu\footnotemark[3]
\and Adrian Sandu\footnotemark[2]
}

\headers{MFEnKF with Reduced Order POD Control Variates}{A. A. Popov, C. Mou, T. Iliescu, and A. Sandu}

\usepackage{csl}

\begin{document}

%\nocite{*}

\csltitle{\realtitle}
\cslauthor{Andrey A. Popov, Changhong Mou, Traian Iliescu, and Adrian Sandu}
\cslyear{20}
\cslreportnumber{2}
\cslemail{apopov@vt.edu, cmou@vt.edu, iliescu@vt.edu, sandu@cs.vt.edu}
\csltitlepage

\maketitle

\begin{abstract} 
This work develops a new multifidelity ensemble Kalman filter (MFEnKF) algorithm based on linear control variate framework.  The approach allows for rigorous multifidelity extensions of the EnKF, where the uncertainty in  coarser fidelities in the hierarchy of models represent control variates for the uncertainty in finer fidelities. 
Small ensembles of high fidelity model runs are complemented by larger ensembles of cheaper, lower fidelity runs, to obtain much improved analyses at only small additional computational costs. 
We investigate the use of reduced order models as coarse fidelity control variates in the MFEnKF, and provide analyses to quantify the improvements over the traditional ensemble Kalman filters. 
    We apply these ideas to perform data assimilation with a quasi-geostrophic test problem, using direct numerical simulation and  a corresponding POD-Galerkin reduced order model. Numerical results show that the two-fidelity MFEnKF provides better analyses than existing EnKF algorithms at comparable or reduced computational costs.

\end{abstract}

% REQUIRED

% REQUIRED

%%%%%%%%%%%%%%%%%%%%%%
%!TEX root = main.tex
%%%%%%%%%%%%%%%%%%%%%%
\section{Introduction.}
%%%%%%%%%%%%%%%%%%%%%%

Data assimilation~\cite{asch2016data,law2015data,reich2015probabilistic} aims to improve forecasting power of dynamical systems~\cite{strogatz2018nonlinear} by fusing information from mathematical models and observations from nature. Ensemble Kalman filters (EnKF)~\cite{evensen2009data,evensen1994sequential,Burgers_1998_EnKF,kalnay2003atmospheric}  have gained widespread popularity for large-scale data assimilation. They use a Monte Carlo approach to propagate covariance information, and take advantage of ensemble forecasting to remove the linear model assumption in conventional Kalman filtering. 

The idea of leveraging a hierarchy of models for increasing the efficiency of Monte Carlo estimation algorithms was proposed in~\cite{Giles_2008_MLMC,Giles_2015_MLMC}, and the multi-level Monte Carlo approach was successfully applied for inference with low-dimensional models.
Recent work extended the multilevel idea to operational EnKF algorithms for stochastic models~\cite{Hoel_2016_MLEnKF,Chernov_2017_MLEnKF,Kikuchi_2015_ROM-EnKF}, proposed multilevel sampling ensemble smoothers\cite{Sandu_2016_reduced-sampling4DVar}, and developed multilevel particle filters \cite{Reich_2016_MLETPF,Gregory_2017_MLETPF} . 

Reduced order modeling is the approach of constructing inexpensive surrogates able to capture the dominant dynamics of large-scale systems.  Previous work has employed 
reduced order models (ROMs)~\cite{brunton2019data,hesthaven2015certified,quarteroni2015reduced} to speed up variational data assimilation~\cite{Yaremchuk_2009_ROM-fdvar,Tian_2011_POD-fdvar,Heemink_2006_ROM-fdvar,Sandu_2015_POD-inverse,daescu2007efficiency,cao2007reduced,kaercher2018reduced,maday2015parameterized}. The underlying idea is to perform the optimization in a reduced order space, and then to reconstruct this subspace around the new point in the full state space. It has been shown in \cite{Sandu_2015_POD-inverse} that the reduced order basis needs to include snapshot information from both the forward and the adjoint models in order for the reduced space optimization to progress to the full order optimal point.

ROMs have also been used to develop new  Markov chain Monte Carlo~\cite{cui2015data,galbally2010non,himpe2015data}, Kalman filters (KF)~\cite{dihlmann2016reduced}, and EnKF~\cite{he2011use,lin2014efficient,pagani2017efficient,xiao2018parameterised} algorithms.
As opposed to variational methods, in the KF and EnKF settings ROMs have been used as replacements to traditional physics-based models. To quantify the effect of replacing the physics-based models with ROMs, rigorous error bounds were derived for both the KF~\cite{dihlmann2016reduced} and EnKF~\cite{pagani2017efficient} algorithms.

This work develops a new multifidelity ensemble Kalman filter algorithm building upon the theory of multivariate control variates~\cite{rubinstein1985efficiency}, and on ROM data assimilation approaches \cite{Sandu_2015_POD-inverse}. 
Small ensembles of high fidelity model runs are complemented by larger ensembles of cheaper, lower fidelity runs, to obtain much improved analyses at only small additional computational costs.
New contributions of this work include rederiving the EnKF data assimilation approach from a multivariate linear control variate theory perspective. This perspective allows for rigorous multifidelity extensions of the EnKF, where the uncertainty in  coarser levels in the hierarchy of models represent control variates for the uncertainty in finer levels. Moreover, the state of different control variates can reside in different spaces (e.g. those with different dimension and/or different inner product), which justifies the ``multifidelity'' \cite{peherstorfer2018survey} EnKF name given to our approach. The mapping between different spaces (i.e., the mapping of each control variate to the space of the corresponding principal variates)  is done by coupling operators that can be computed in an optimal way.  The paper derives an optimal statistical estimation framework in order to show significant reduction in both the cost of the method and in error. 

Key innovations of the multifidelity EnKF approach as compared to the standard multi-level EnKF (MLEnKF) proposed in \cite{Hoel_2016_MLEnKF,Chernov_2017_MLEnKF} include the use of multivariate linear control variate theory \cite{rubinstein1985efficiency} to rigorously incorporate all model levels in the statistical estimation approach. MLEnKF \cite{Hoel_2016_MLEnKF,Chernov_2017_MLEnKF} incorporates different model levels using signed empirical measures, which makes the multilevel covariances possibly non-positive; in our approach the multilevel empirical covariances are always non-negative. The use of signed empirical measures over the fine space requires MLEnKF to represent the states from all model levels into the same (fine level) space. In the proposed multifidelity EnKF approach different control variates represent states from different model levels that live in their own spaces; the mapping between spaces is done explicitly by coupling operators that are computed in a statistically optimal manner, obtained from the theory to the optimal gains required to compute the covariance estimates. (We note that these operators are equal to identity in MLEnKF).

The remainder of the paper is organized as follows. 
The data assimilation problem, control variate theory, and the ensemble Kalman filter are reviewed in \cref{sec:background,}.
Properties of the ROM as a control variate are analyzed in \cref{sec:podcontrol}.
The multifidelity ROM EnKF algorithm is derived in \cref{sec:twospacepodenkf}.
The quasi-geostrophic test problem and the corresponding models of different fidelity are introduced in \cref{sec:models}.
Numerical experiments are reported in \cref{sec:numerical}, 
and closing remarks are made in \cref{sec:conclusions}.

%%%%%%%%%%%%%%%%%%%%%%
%!TEX root = main.tex
%%%%%%%%%%%%%%%%%%%%%%
\section{Control variates and the Ensemble Kalman filter.}
\label{sec:background}
%%%%%%%%%%%%%%%%%%%%%%

Consider the data assimilation problem of predicting the state of a natural phenomenon through modeling and sparse noisy nonlinear observations. 

Let $\xa_{i-1}$ be a random variable whose distribution represents (our uncertain knowledge of) the true state of the physical system, projected onto model space, at time $t_{i-1}$. This knowledge is propagated to time $t_{i}$ through the model dynamics $\Model$:
\begin{equation}
\label{eqn:umodel}
    \xb_i = \Model_{i-1,i}\left(\xa_{i-1},\mathcal{E}^{\Model}_i\right) \stackrel{\rm assumed}{=} \Model_{i-1,i}\left(\xa_{i-1}\right),
\end{equation}
where the distribution of the random variable $\xb_i$ represents the prior knowledge of the state as time $i$, and $\mathcal{E}^{\Model}_i$ is a random variable quantifying stochastic effects and model errors. In this paper we assume that the model \cref{eqn:umodel} is deterministic and exact, meaning that $\mathcal{E}^{\Model}_i = 0$.
Noisy observations of the true state $\xt_i$ are collected at time $t_i$:
\begin{equation}
\label{eqn:uobs}
    \Obs_i = \Hobs_i\left(\xt_i,\mathcal{E}^{\Hobs}_i\right)  \stackrel{\rm assumed}{=}   \Hobs_i(\xt_i) + \erro_i, \quad \erro_i\sim\!N(\*0,\Cov{\erro_i}{\erro_i}),
\end{equation}
where $\Hobs_i$ is the (non)linear observation operator, and $\mathcal{E}^{\Hobs}_i$ is a random variable quantifying uncertainty in the observations. We assume that the observation errors $\erro_i$ are additive, unbiased, and Gaussian, with observation error covariance matrices $\Cov{\erro_i}{\erro_i}$.

% COMMENT: I want to avoid using the term `a priori' and instead use prior, as `a priori' is antithetical to the ideal of the Kalman filter for non-linear systems, and all of Bayesian reasoning, as it implies a truth, and not merely knowledge.
% Actually they are used intercheangeably and mean the same thing in statistics.
% They are used interchagibly, but they do not have the same origin. While often times they are used to mean the same thing, `a priori' historically refers to true knowledge. It's a minor point, I know, but I am that way, haha.

Using the prior knowledge of the state \cref{eqn:umodel} at time $t_i$ described by the probability density  $\pi(\xb_{i})$, and the likelihood of observations \cref{eqn:uobs} described by the probability density  $\pi(\Obs_i|\xb_{i})$, the Bayesian approach gives the posterior knowledge of the state:
\begin{equation}
    \pi(\xb_{i}|\Obs_i) \propto \pi(\Obs_i|\xb_{i})\,\pi(\xb_{i}).
\end{equation}
We seek to approximate this posterior probability density in an ensemble Kalman filter framework.

As some concepts in this paper are not present in traditional data assimilation literature, we use the following notation (slightly different from ~\cite{ide1997unified}) for presentation clarity. Let $\chi$ and $\upsilon$ be random variables. The exact mean of $\chi$ is denoted $\Mean{\chi}$, and the empirical (sample) mean by $\MeanE{\chi}$. Similarly, the exact covariance is denoted by $\Cov{\chi}{\upsilon}$, and the sample covariance  by $\CovE{\chi}{\upsilon}$. An ensemble of samples of $\chi$ is denoted by $\En{\chi}$, the $i$-th sample by $\*\chi^{[i]}$, and the scaled ensemble anomalies (defined later) by $\An{\chi}$. 

%%%%%%%%%%%%%%%%%%%
\subsection{Linear control variates.}
%%%%%%%%%%%%%%%%%%%

Consider a random variable $\chi$ with support $\Omega_\chi = \mathbb{R}^n$ with a distribution that represents the uncertainty in the state. Its mean $\Mean{\chi}$ represents the minimum variance estimator of the true state, and its higher moments quantify the confidence in this estimator. We call $\chi$ the {\it principal variate}.

Consider a second random variable $\hat\upsilon$ with support $\Omega_{\hat\upsilon}=\mathbb{R}^r$, which is highly correlated with $\chi$, and has a known mean $\Mean{\hat\upsilon}$. 
This second random variable $\hat\upsilon$ is a \textit{control variate} used to improve the estimate $\Mean{\chi}$ of the true state of the system. Here we consider $r\ll n$, though this is not required in general. 

Our goal is to estimate $\Mean{\chi}$, and the direct way is to sample the principal variate $\chi$. The linear control variate approach seeks to obtain better estimates by taking samples of {\it both} the principal variate $\chi$ and the control variate $\hat\upsilon$. Specifically, in a multivariate linear control variate approach~\cite{rubinstein1985efficiency} one constructs the new random variable
\begin{equation}
\label{eqn:linearcv}
    \zeta = \chi - \*S\,(\hat{\upsilon} - \Mean{\hat\upsilon}),
\end{equation}
that we call the \textit{total variate}, which has the same mean as the principal variate $\Mean{\zeta}=\Mean{\chi}$, but whose other moments have been modified by the knowledge of the control variate $\hat\upsilon$. The deterministic {\it gain matrix} $\,\*S\,\in\mathbb{R}^{n\times r}$ is chosen such as to minimize the generalized variance of the total variate. We recall the following result \cite[Lemma 1 in Appendix]{rubinstein1985efficiency}:
\begin{lemma}[Optimal gain]
\label{lem:linearcv}
The optimal gain that minimizes the generalized variance of the total variate $\zeta$ (the determinant of $\Cov{\zeta}{\zeta}$) is
\begin{equation}
 \label{eqn:optimallinearcv}
    \*S = \Cov{\chi}{\hat{\upsilon}}\,\Cov{\hat{\upsilon}}{\hat{\upsilon}}^{-1}.
\end{equation}
\end{lemma}
%
% \begin{proof}
% The proof is similar to the one in~\cite{rubinstein1985efficiency}. Assume that the optimal gain is of the form $\,\*S = \Cov{\chi}{\hat{\upsilon}}\Cov{\hat{\upsilon}}{\hat{\upsilon}}^{-1} + \*D$, where $\*D$ is any arbitrary real matrix in $\mathbb{R}^{n\times r}$. We have
% %
% \begin{equation*}
%     \Cov{\zeta}{\zeta} = \left(\Cov{\chi}{\chi} - \Cov{\chi}{\hat\upsilon}\Cov{\hat\upsilon}{\hat\upsilon}^{-1}\Cov{\hat\upsilon}{\chi}\right) + \*D\Cov{\hat\upsilon}{\hat\upsilon}\*D^\intercal,
% \end{equation*}
% %
% which is the sum of two symmetric semi-positive definite matrices. The determinant of the sum is minimized when $\*D=\*0$, as required.
% \end{proof}

Consider now the case where the mean $\Mean{\hat\upsilon}$ of the control variate is unknown. However, one can sample a random variable $\upsilon \in \Re^r$ that has the same mean and support as $\hat\upsilon$ but is independent of both $\chi$ and $\hat\upsilon$.  Using (a realization of) what we call the \textit{ancillary variate} $\upsilon$ as a proxy for the exact mean $\Mean{\hat\upsilon}=\Mean{\upsilon}$, the total variate  \cref{eqn:linearcv} becomes:
\begin{equation}
\label{eqn:CV-independent-mean}
    \zeta = \chi - \,\*S\,(\hat{\upsilon} - \upsilon).
\end{equation}
Letting $\omega = \hat{\upsilon} - \upsilon + \Mean{\hat\upsilon}$, $\Mean{\omega} = \Mean{\hat\upsilon}$, and $\hat{\upsilon} - \upsilon = \omega - \Mean{\omega}$, equation~\eqref{eqn:CV-independent-mean} reduces to equation~\cref{eqn:linearcv} with $\Mean{\hat\upsilon}$ replaced by $\omega$. By \cref{lem:linearcv} the optimal gain is:
\begin{equation}
    \label{eqn:optimal-gain}
    \,\*S = \Cov{\chi}{\hat{\upsilon}}\,{\left(\Cov{\hat{\upsilon}}{\hat{\upsilon}} + \Cov{\upsilon}{\upsilon}\right)}^{-1},
\end{equation}
where $\Cov{\hat{\upsilon}}{\hat{\upsilon}} + \Cov{\upsilon}{\upsilon} = \Cov{\omega}{\omega}$ 
and $\Cov{\chi}{\hat\upsilon} = \Cov{\chi}{\omega}$.

If the control variate $\hat\upsilon$ and its ancillary variate $\upsilon$ share not only the same mean, but also the same covariance, $\Cov{\upsilon}{\upsilon}= \Cov{\hat{\upsilon}}{\hat{\upsilon}}$, then~\eqref{eqn:optimal-gain} becomes 
\begin{equation}
   \label{eqn:optimal-gain-half}
    \,\*S = \frac{1}{2}\Cov{\chi}{\hat{\upsilon}}\Cov{\hat{\upsilon}}{\hat{\upsilon}}^{-1}.
\end{equation}
%
% This assumption, in not necessarily true in general. In practice, if we suppose the covariances of $\hat\upsilon$ and $\upsilon$ are scalings of each other, then an approximation of the optimal gain is:
% \blue{refs?}
% %
% \begin{equation}
%   \label{eqn:optimal-gain-s}
%     \,\*S \approx s\,\Cov{\chi}{\hat{\upsilon}}\Cov{\hat{\upsilon}}{\hat{\upsilon}}^{-1},
% \end{equation}
% %
% where $s$ is a scaling factor either predetermined or derived adaptively. Choices of $s$ will be discussed in \cref{sec:twospacepodenkf}.

\begin{remark}[Total variate covariance]\label{rem:positivity-of-covariance}
The covariance of the total variate \cref{eqn:CV-independent-mean} using the optimal gain \cref{eqn:optimal-gain} is:
\begin{equation}
\label{eqn:covariance-total-variate}
\begin{split}
    \Cov{\zeta}{\zeta} &=  \Cov{\chi}{\chi} - \Cov{\chi}{\hat\upsilon}\,\*S^\intercal - \,\*S\,\Cov{\hat\upsilon}{\chi} + \,\*S\,\Cov{\hat\upsilon}{\hat\upsilon}\,\*S^\intercal + \,\*S\,\Cov{\upsilon}{\upsilon}\,\*S^\intercal\\
    &= \Cov{\chi}{\chi} - \Cov{\chi}{\hat\upsilon}{\left(\Cov{\hat\upsilon}{\hat\upsilon} + \Cov{\upsilon}{\upsilon}\right)}^{-1}\Cov{\hat\upsilon}{\chi}.
\end{split}
\end{equation}
Note that this is always a symmetric semi-positive definite (s.p.d.) matrix that is smaller (in s.p.d. matrix sense) than the principal variate covariance, $0 \le \Cov{\zeta}{\zeta} \le \Cov{\chi}{\chi}$. In contrast, the multilevel covariance formula for variables that live in the same space, $\Cov{\zeta}{\zeta} = \Cov{\chi}{\chi} - \Cov{\hat{\upsilon}}{\hat{\upsilon}} + \Cov{\upsilon}{\upsilon}$ proposed in~\cite{Hoel_2016_MLEnKF}, does not necessarily enjoy these properties, as the signed empirical measure ignores cross covariances. The covariance \cref{eqn:covariance-total-variate} is s.p.d. for any matrix $\,\*S$ by the construction in the proof of \cref{lem:linearcv}.
\end{remark}

%%%%%%%%%%%%%%%%%%%%%%
\subsection{Multiple fidelities of control variates.}
\label{sec:multifidelity}
%%%%%%%%%%%%%%%%%%%%%%
One can recursively apply  the control variate approach \eqref{eqn:CV-independent-mean} to improve estimation of the mean $\Mean{\hat{\upsilon}}$. To this end, in \eqref{eqn:CV-independent-mean} we identify $\upsilon_0 \equiv \chi $ and the first fidelity the control and ancillary variate with $\hat{\upsilon}_1 \equiv \hat{\upsilon} $ and $\upsilon_1 \equiv \upsilon$, respectively. Next, we consider $\upsilon_1$ as a principal variate, and use a control variate $\hat{\upsilon}_2$ to build a total variate $\upsilon_1 - \,\*S_2\,(\hat{\upsilon}_2-\Mean{\hat{\upsilon}_2})$. Next, replace the exact mean $\Mean{\hat{\upsilon}_2}$ by a realization of  the ancillary variate $\upsilon_2$, and repeat until we reach the $\!L$-th fidelity control variate $\hat{\upsilon}_{\!L}$ with the ancillary variate $\upsilon_{\!L}$. This telescopic structure replaces the ancillary variate at fidelity $\ell-1$ by a total variate constructed using the next fidelity $\ell$ control and ancillary variates:
\begin{equation}
\label{eqn:CV-independent-mean-fidelity}
    \upsilon_{\ell-1} \xleftarrow{} \upsilon_{\ell-1} - \,\*S_\ell(\hat{\upsilon}_\ell - \upsilon_\ell),
\quad  \,\*S_\ell = \Cov{\upsilon_{\ell-1}}{\hat{\upsilon}_\ell}{\left(\Cov{\hat{\upsilon}_\ell}{\hat{\upsilon}_\ell} + \Cov{\upsilon_\ell}{\upsilon_\ell}\right)}^{-1}, 
\quad \ell = 1, \dots, \!L.
\end{equation}
The total variate $\zeta$, representing a multifidelity control variate approach for the top fidelity principal variate $\chi$, is:
\begin{equation}
    \zeta = \chi - \sum_{\ell=1}^{\!L} \overline{\,\*S}_\ell\,\left(\hat{\upsilon}_{\ell} - \upsilon_{\ell}\right),
    \quad \overline{\,\*S}_\ell = \prod_{\lambda=1}^{\ell}\,\*S_{\lambda}.
%    \quad \,\*S_{1} \equiv \,\*S~\textnormal{from}~\eqref{eqn:optimal-gain}.
\end{equation}
%

%If $\chi$ and $\hat{\upsilon}$ are defined on a separable Hilbert space in terms of the orthogonal basis $\*\Phi$, such that $\chi = \sum_{i=1}^\infty\chi_i\*\phi_i$, and $\hat{\upsilon} = \sum_{i=1}^r\chi_i\*\phi_i$ with $r\leq\infty$, and $\Cov{\upsilon}{\upsilon} = \Cov{\hat\upsilon}{\hat\upsilon}$, then

%

%%%%%%%%%%%%%%%%%%%%%%%%%%%%
\subsection{Implementation of linear control variates using ensembles.}
%%%%%%%%%%%%%%%%%%%%%%%%%%%%

In practice, the exact distributions of $\chi$, $\hat\upsilon$, and $\upsilon$ are not available, therefore computing the exact moments of the total variate $\zeta$ is not possible. However, we assume that one can sample from these distributions, and seek to estimate the statistics of $\zeta$.

For this, take $\Nens_\chi$ pairwise samples $(\*\chi^{[k]}, \hat{\*\upsilon}^{[k]})$ of the principal and control variates (to be able to derive correlated statistics), and construct the ensembles $\En{\chi} =$ $\left[\*\chi^{[1]}, \dots, \*\chi^{[\Nens_\chi]}\right]$ $\in \Re^{n \times \Nens_\chi}$ and
$\En{\hat\upsilon} =$ $\left[\hat{\*\upsilon}^{[1]}, \dots, \hat{\*\upsilon}^{[\Nens_\chi]}\right]$ $\in \Re^{r \times \Nens_\chi}$.
Take $\Nens_\upsilon$ samples $\*\upsilon^{[k]}$ of the ancillary variate and construct the ensemble  $\En{\upsilon} = \left[\*\upsilon^{[1]}, \dots, \*\upsilon^{[\Nens_\upsilon]}\right]$ $\in \Re^{r \times \Nens_\upsilon}$. 

The empirical means and the ensembles of anomalies are defined as: 
\begin{equation}
\begin{gathered}
    \MeanE{\chi} \coloneqq \Nens_\chi^{-1}\,\En{\chi}\,\*1_{\Nens_\chi},\quad 
    \MeanE{\hat{\upsilon}} \coloneqq \Nens_\chi^{-1} \,\En{\hat\upsilon}\,\*1_{\Nens_\chi},\quad 
    \MeanE{\upsilon} \coloneqq \Nens_{\upsilon}^{-1} \,\En{\upsilon}\,\*1_{\Nens_{\upsilon}},\\
    \An{\chi} \coloneqq (\Nens_\chi-1)^{-\frac{1}{2}} \left(\En{\chi} - \MeanE{\chi}\*1_{\Nens_\chi}^\intercal\right),\
    \An{\hat{\upsilon}} \coloneqq (\Nens_\chi-1)^{-\frac{1}{2}} \left(\En{\hat{\upsilon}} - \MeanE{\hat{\upsilon}}\*1_{\Nens_\chi}^\intercal\right),\\
    \An{\upsilon} \coloneqq (\Nens_\upsilon-1)^{-\frac{1}{2}} \left(\En{\upsilon} - \MeanE{\upsilon}\*1_{\Nens_\upsilon}^\intercal\right),
\end{gathered}
\end{equation}
which leads to the empirical covariances:
\begin{equation}
\begin{gathered}
    \CovE{\chi}{\chi} =\An{\chi}\An{\chi}^\intercal,\quad 
    \CovE{\hat{\upsilon}}{\hat{\upsilon}} =\An{\hat{\upsilon}}\An{\hat{\upsilon}}^\intercal,\quad
    \CovE{\chi}{\hat{\upsilon}} = \An{\chi}\An{\hat{\upsilon}}^\intercal = 
    \CovE{\hat{\upsilon}}{\chi }^\intercal,\quad
    \CovE{\upsilon}{\upsilon} =\An{\upsilon}\An{\upsilon}^\intercal.
\end{gathered}
\end{equation}

The empirical mean and covariance estimates of the total variate \cref{eqn:CV-independent-mean} are:
\begin{equation}
\begin{split}
    \MeanE{\zeta} &= \Nens_\chi^{-1} \sum_{k=1}^{\Nens_\chi} \left(\*\chi^{[k]} - \,\*S \hat{\*\upsilon}^{[k]}\right) +  \Nens_\upsilon^{-1}\sum_{k=1}^{\Nens_\upsilon} \,\*S\,\*\upsilon^{[k]}, \\
%
% which has the covariance,
% %
% \begin{equation}
% \label{eqn:meanestvar}
%     \Cov{\MeanE{\zeta}}{\MeanE{\zeta}} = \Nens_\chi^{-1}\left(\Cov{\chi}{\chi} + \,\*S\,\,\Cov{\hat{\upsilon}}{\hat{\upsilon}}\,\,\*S^\intercal - \Cov{\chi}{\hat{\upsilon}}\,\,\*S^\intercal - \,\*S\,\,\Cov{\hat{\upsilon}}{\chi}\right)  + \Nens_\upsilon^{-1}\,\*S\,\,\Cov{\upsilon}{\upsilon}\,\,\*S^\intercal,
% \end{equation}
% %
% that  is minimized for the same value of the optimal gain $\,\*S$ as $\Cov{\zeta}{\zeta}$.
%
    \CovE{\zeta}{\zeta} &= \CovE{\chi}{\chi} + \,\*S\,\CovE{\hat{\upsilon}}{\hat{\upsilon}}\,\*S^\intercal - 
    \CovE{\chi}{\hat{\upsilon}}\,\*S^\intercal - \,\*S\,\CovE{\hat{\upsilon}}{\chi} +  \,\*S\,\CovE{\upsilon}{\upsilon}\,\*S^\intercal.
\end{split}    
\end{equation}

When the exact covariances $\Cov{\chi}{\hat{\upsilon}}$ and $\Cov{\hat{\upsilon}}{\hat{\upsilon}}$ are not known, but the exact covariance of the ancillary variate $\Cov{\upsilon}{\upsilon}$ is known,  the optimal gain matrix \cref{eqn:optimal-gain} is approximated by
\begin{equation}
    \,\*S \approx \,\widetilde{\*S} = \CovE{\chi}{\hat{\upsilon}}{\left(\CovE{\hat{\upsilon}}{\hat{\upsilon}} + \Cov{\upsilon}{\upsilon}\right)}^{-1}, \label{eqn:encvenkf}
\end{equation}
which is well defined when $\Cov{\upsilon}{\upsilon}$ is full rank. In the case where the underlying random variables are Gaussian, the expected value of the sampled gain matrix, $\*{\tilde{S}}$, is not the exact gain matrix, even in the scalar case\cite{popov2020explicit}. 

When $\Cov{\upsilon}{\upsilon}$ is also unknown, and all empirical covariance estimates are undersampled, meaning that the rank of the sampled covariance is lower than the rank of the covariance of the underlying random variable, then the approximation
\begin{equation}
    \,\*S \approx \,\widetilde{\*S} = \CovE{\chi}{\hat{\upsilon}}{\left(\CovE{\hat{\upsilon}}{\hat{\upsilon}} + \CovE{\upsilon}{\upsilon}\right)}^{-1},\label{eqn:sgainfullest}
\end{equation}
can be ill-defined, and a better approach is required to estimate the optimal gain matrix. In this case our goal will be to determine a control variate whose relation with the principal variate leads to a good approximation of the gain matrix with minimal reliance on sampling.

If the cost of obtaining one sample of the principal variate  is $C_\chi$ and the cost of a sample from either the ancillary or the control variates is $C_\upsilon$, then the cost of a two fidelity estimator is:
\begin{equation}
    \Nens_\chi C_\chi + (\Nens_\chi + \Nens_\upsilon) C_\upsilon,\label{eqn:cost}
\end{equation}
which, if the cost of sampling the coarser random variables is negligible $C_\upsilon \ll C_\chi$, is roughly equal to the cost of sampling the principal variate. 

%%%%%%%%%%%%%%%%%%%%%%%%
\subsection{Ensemble Kalman filter.}
%%%%%%%%%%%%%%%%%%%%%%%%

The traditional Kalman filter~\cite{Kalman_1960} (KF) aims to optimally solve the Bayesian inference problem, under the assumption that the probability distributions of the prior knowledge about the state, observations, and the resulting posterior knowledge are all Gaussian. The KF also makes the assumptions that $\Mean{\xb_i} = \xt_i$ and $\Mean{\Hobs(\xb_i)} = \Mean{\Obs_i}$.  We now re-derive the ensemble Kalman filter (EnKF) framework from a multivariate linear control variate theory perspective.

The principal variate represents our prior knowledge $\chi  \equiv \xb_i$, the control variate is the model-predicted observations $\hat{\upsilon} \equiv \Hobs(\xb_i)$, and the ancillary variate is the observations $\upsilon \equiv \Obs_i$. The goal is to estimate the true state, which is the mean of the principal variate $\Mean{\chi} \equiv \Mean{\xb_i} = \xt_i$. The posterior knowledge is represented by the new, reduced variance total variate $\zeta \equiv \xa_i$ \eqref{eqn:CV-independent-mean}:
\begin{equation}
\label{eqn:KF}
    \xa_i = \xb_i - \*K_i\left(\Hobs(\xb_i) - \Obs_i\right),
\end{equation}
where the control variate gain matrix $\*S \equiv \*K_i$ is the Kalman gain. The mean of the total variate is also the true state $\Mean{\zeta} \equiv \Mean{\xa_i} = \xt_i = \Mean{\chi}$, but its covariance is smaller.

EnKF represents the random variables by ensembles of $\Nens$ samples, with $\En{\xb}$ and $\En{\Hobs(\xb)}$ defined in the usual way.
The perturbed observations version of the EnKF~\cite{Burgers_1998_EnKF} also constructs an ensemble of independent samples from the observation distribution: 
\begin{equation}
    \En{\Obs_i} = \obs_i\,\*1^\intercal_{\Nens} + \An{\erro_i}, 
\end{equation}
where the anomalies $\An{\erro_i}$ are derived from an ensemble of $\Nens$ independent samples from the observation error distribution \eqref{eqn:uobs}.

% Seems to contradict what we say in the first paragraph in this section; needs more explaining to make it "right".
%\begin{remark}[Violation of CV structure]
%Note that, we violate the assumption that the control variate $\Hobs(\xb_i)$ and the independent mean estimator $Y_i$ have the same mean. This is known to not be a practical issue  for convergence~\cite{LeGland_2009_convergence-EnKF,popov2020explicit}.
%\end{remark}

It is typically assumed that the only variable whose covariance is known is $\Obs_i$, meaning that the Kalman gain is approximated using~\cref{eqn:encvenkf},
\begin{equation}
\label{eqn:statKgain}
    \widetilde{\mathbf{K}}_i = \CovE{\xb_i}{\Hobs(\xb_i)}{\left(\CovE{\Hobs(\xb_i)}{\Hobs(\xb_i)} + \Cov{\erro_i}{\erro_i}\right)}^{-1}.
\end{equation}
Thus, the EnKF analysis formulas are:
\begin{equation}
\label{eqn:enkfanalysis}
\begin{gathered}
    \MeanE{\xa_i} = \MeanE{\xb_i} - \widetilde{\mathbf{K}}_i\,\widetilde{\mathbf{d}}_i,\quad \An{ \xa_i } = \An{ \xb_i } - \widetilde{\mathbf{K}}_i\left(\An{ \Hobs(\xb_i) } - \An{\erro_i}\right),\\
    \widetilde{\mathbf{d}}_i = \MeanE{\Hobs(\xb_i)} - \MeanE{\Obs_i},\quad \En{\xa_i} = \MeanE{\xa_i}\,\one_\Nens^\intercal + (\Nens-1)^{\frac{1}{2}}\,\An{ \xa_i },
\end{gathered}
\end{equation}
with the ensemble $\En{\xa_i}$ representing the posterior uncertainty at time $t_i$.

The number of ensemble members is usually significantly smaller than the dimension of the state space, $N \ll n$, and the covariance matrix estimate is affected by sampling errors. In order to alleviate these errors, and probabilistically inaccurate assumptions about the statistical Kalman gain \cref{eqn:statKgain}, methods such as inflation~\cite{Anderson_2001_EAKF,tong2015nonlinear,anderson2007adaptive,popov2020explicit}, localization~\cite{popov2019bayesian,Anderson_2012_localization,petrie2008localization,Sandu_2018_Covariance-Cholesky}, and covariance shrinkage~\cite{Sandu_2015_covarianceShrinkage,Sandu_2019_Covariance-parallel, popov2020stochastic,Sandu_2017_parallel-EnKF} have been developed. 
%
\iffalse
One variant of covariance inflation, implemented by rescaling the anomalies
%
\begin{equation}
    \An{\xb_i} \xleftarrow{} \alpha_{\xb_i}\,\An{\xb_i}, \quad \alpha_{\xb_i} > 1,
\end{equation}
%
is applied before computing the gain in order to attempt to prevent ensemble collapse towards the mean. The nonlinear EnKF transformation of the anomalies \eqref{eqn:enkfanalysis} also induces a non-Gaussian distribution on prior Gaussian samples that for a finite ensemble, in the case of an ideal square-root scalar filter, requires corrections like inflation~\cite{popov2020explicit}, in order to have the same average behavior as the Kalman filter.
%
Localization and covariance shrinkage methods on the other hand aim to enhance the rank and independence of the statistical covariance estimate. We will not explore these methods in this paper, other than as numerical references.
\fi
%%%%%%%%%%%%%%%%%%%%%%
%!TEX root = main.tex
%%%%%%%%%%%%%%%%%%%%%%
\section{Spaces, projections, information, and control variates.}
\label{sec:podcontrol}
%%%%%%%%%%%%%%%%%%%%%%

Bayes' rule requires to use all information  information in the inference process \cite{jaynes2003probability}; in particular, if additional information about the dynamics of the system is known, it must be used in the inference in order to increase confidence in the inference results.  The assumption of linearity (in KF and linear control variates), however, precludes the inclusion of important information about the manifold on which nonlinear model dynamics live. Reduced order models (ROMs) construct linear subspaces that capture the most important (in some well-defined sense) features and modes of the full order dynamics. For this reason we seek to build enhanced ensemble Kalman filters with ROMs as control variates.  

To this end we consider finite dimensional random variables. Without loss of generality, the principal variate $\chi$ lives in the space $\mathcal{S}_\chi$ = $\mathbb{R}^n$, endowed with the canonical Euclidean basis and the canonical Euclidean inner product $\inner{\cdot}{\cdot}_{\mathcal{S}_\chi}$ = $\inner{\cdot}{\cdot}_{\Re^n}$. The control and ancillary variates ($\hat\upsilon$ and $\upsilon$, respectively) are vectors in $\mathcal{S}_\upsilon$ = $\mathbb{R}^r$,  endowed with the  canonical Euclidean basis and the canonical inner product $\inner{\cdot}{\cdot}_{\mathcal{S}_\upsilon}$ = $\inner{\cdot}{\cdot}_{\Re^r}$.

We consider the natural idea of utilizing a control variate that is the projection of the principal variate $\chi \in \mathcal{S}_\chi$ onto an $r$-dimensional subspace $\widehat{\mathcal{S}}_\upsilon \subset \mathcal{S}_\chi$ that captures the dominant features of the nonlinear dynamics of the system. 

We identify the space of control and ancillary variates $\mathcal{S}_\upsilon$ =  $\mathbb{R}^r$ with an $r$-dimensional subspace $\widehat{\mathcal{S}}_\upsilon \subset \mathcal{S}_{\chi}$ equipped with the $\inner{\cdot}{\cdot}_{\widehat{\mathcal{S}}_\upsilon}$ = $\inner{\cdot}{\cdot}_{\*M_{\upsilon}}$ inner product, where $\*M_{\upsilon} \in \Re^{n \times n}$ is a s.p.d. matrix. Specifically, let $\*\Phi = [\Phi_1,\dots, \Phi_n]  \in\mathbb{R}^{n\times n}$ be an $\*M_{\upsilon}$-orthogonal basis of $\Re^n$; we identify the control space with the span of the first $r$ vectors in the basis $\widehat{\mathcal{S}}_\upsilon = \textnormal{Span}\{\Phi_1,\dots, \Phi_r\}$. 
Consider two vectors in the control space $\*u,\*v \in \Re^r$; their representations as $n$-dimensional vectors in $\mathcal{S}_\upsilon$ are $\*\Phi_r\,\*u$ and $\*\Phi_r\,\*v$, respectively, where $\*\Phi_r = [\Phi_1,\dots, \Phi_r]  \in\mathbb{R}^{n\times r}$. The dot-product is preserved in both representations:
\begin{equation}
    \inner{\*u}{\*v}_{\Re^r} = \*u^{\intercal}\,\*v = \*u^{\intercal}\,\*\Phi_r^{\intercal}\,\*M_{\upsilon}\,\*\Phi_r\,\*v
   = \inner{\*\Phi_r\,\*u}{\*\Phi_r\,\*v}_{\*M_{\upsilon}}.
\end{equation}
\begin{remark}
There is no loss of generality with the  above formulation. Consider the control space $\mathcal{S}_\upsilon$ =  $\Re^r$ endowed with the general inner product $\inner{\cdot}{\cdot}_{\*N_\upsilon}$, and the $\*N_\upsilon$-orthonormal basis $\{\varphi_1,\dots,\varphi_r\}$. Identify the control space with the $r$-dimensional subspace $\widehat{\mathcal{S}}_\upsilon = \operatorname{span}\{ \widehat{\Phi}_1,\dots, \widehat{\Phi}_r\} \subset \Re^n$, where $\widehat{\*\Phi}_r = [\widehat{\Phi}_1,\dots, \widehat{\Phi}_r]  \in\mathbb{R}^{n\times r}$ are the control basis vectors represented as vectors in $\Re^n$. The following change of basis casts this general case in our formulation:
\[
\*\Phi_r = \*M_{\upsilon}^{-\frac{1}{2}}\,\widehat{\*\Phi}_r\,\*N_\upsilon^{-\frac{1}{2}}, \quad \*\Phi_r^\intercal\, \*M_\upsilon\,\*\Phi_r = \mathbf{I}_{r \times r},
\quad \textnormal{Range}(\*\Phi_r) = \textnormal{Range}(\widehat{\*\Phi}_r) = \widehat{\mathcal{S}}_\upsilon,
\]
where $\*M = \*M^{\frac{1}{2}}\,\*M^{\frac{\intercal}{2}}$ is a square root factorization of the s.p.d. matrix $\*M$.
\end{remark}

\begin{remark}
The transformed vectors $\*\Psi = \*M_{\upsilon}^{\frac{\intercal}{2}}\,\*\Phi$ form an orthonormal basis of $\mathcal{S}_\chi$, and the first $r$ vectors of $\*\Psi$ form an orthonormal basis of $\widehat{\mathcal{S}}_\upsilon$ w.r.t. the Euclidian dot-product:
\[
\*\Psi_r = \*M_{\upsilon}^{\frac{\intercal}{2}}\,\*\Phi_r\, \quad \*\Psi_r^\intercal\, \*\Psi_r = \mathbf{I}_{r \times r},
\quad \textnormal{Range}(\*\Psi_r) = \textnormal{Range}(\*\Phi_r) = \widehat{\mathcal{S}}_\upsilon.
\]
\end{remark}
In summary, a control vector $\*u \in \Re^r$ is represented in the principal space $\Re^n$ as:
\begin{subequations}
\label{eqn:projection-operators}
\begin{equation}
\*x  = \*\Phi_r\, \*u. % = \*M_{\upsilon}^{-\frac{\intercal}{2}}\,\*\Psi_r\, \*u.
\end{equation}
Viceversa, a vector in the principal space $\*x \in \Re^n$ is projected $\*M_{\upsilon}$-orthogonally onto the control space $\Re^r$ as follows:
\begin{equation}
\*\Phi_r^\ast \coloneqq \*\Phi_r^\intercal\,\*M_{\upsilon}, \quad 
\*u = \*\Phi_r^\ast\,\*x = \*\Phi_r^\intercal\,\*M_{\upsilon}\,\*x. % =\*\Psi_r^\intercal\,\*M_{\upsilon}^{\frac{\intercal}{2}}\,\*x.
\end{equation}
\end{subequations}
We note that $\*\Phi_r^\ast  \in \Re^{r \times n}$ is the adjoint operator of $\*\Phi_r \in \Re^{n \times r}$ with respect to the control dot-products: $\langle \*\Phi_r^\ast \*x, \*u \rangle_{\Re^r} = \langle \*x, \*\Phi_r \*u \rangle_{\*M_{\upsilon}}$

\begin{remark}
The method of snapshots \cite{sirovich1987turbulence1} that underpins the ROM finds an $\*M_{\upsilon}$-orthonormal basis $\Phi_1,\dots, \Phi_n$ of  $\Re^n$ with vectors sorted in decreasing order of importance (e.g., with respect to the energy of the dynamical system solution projected onto that vector). In the method of snapshots $\Phi_i$ are the eigenvectors of the temporal covariance of the dynamics (discretely approximated by the snapshot covariance). A full-state vector $\*x = \sum_{i=1}^n u_i \Phi_i \in \Re^n$  is given a reduced order approximation by keeping only the main $r$ components: $\*x_r = \sum_{i=1}^r u_i \Phi_i$. This is equivalent to projecting the vector $\*M_{\upsilon}$-orthonormally onto the first $r$ basis vectors, $\*u = [u_1 \dots u_r]^\intercal = \*\Phi_r^\ast\,\*x$.
\end{remark}

%Thus, the method of snapshots in the finite case produces the matrix $\*\Phi_r$ ($\*\Phi_r^*$) which is both a  projection from (to) Euclidean space into (from) the space of the surrogate inner product and a truncation of the basis into one that represents the $r$ principal modes of the temporal covariance of the dynamical system in said inner product space.
%Our key contribution to the theory of multilevel modeling is the derivation of static optimal gains for such control variates. 

Consider the case where the ensemble size of the principal variate-control variate pair is insufficient to accurately determine their statistical covariances. In this case one cannot accurately determine the statistical analogue of the optimal gain~\cref{eqn:sgainfullest} at any given point in time. To overcome this difficulty we leverage the projection operators defined in this section in order to describe both the control variate and the corresponding optimal gain.

\begin{theorem}
\label{thm:optimalavggain}
Let the control variate \eqref{eqn:linearcv} be the projection of the principal variate over $\widehat{\mathcal{S}}_\upsilon$, 
\begin{equation}
\label{eqn:projected-control}
\hat\upsilon = \*\Phi_r^*\,\chi. 
\end{equation}
The principal and control variates in the $\*\Phi$ basis read:
\begin{equation}
\label{eqn:variates-in-Phi}
%\chi =\sum_{i=1}^n c_i \, \Phi_i 
%= \*\Phi_{:,1:r}\, \*c_{1:r} + \*\Phi_{:,r+1:n}\, \*c_{r+1:n} =  \*\Phi_r\,\hat\upsilon  + \Delta{\chi}_r, \quad
%\hat\upsilon = \*c_{1:r}.   
\chi = \*\Phi_r\, \*\Phi_r^*\, \chi + (\mathbf{I} -  \*\Phi_r\, \*\Phi_r^*) \, \chi = \*\Phi_r\,\hat\upsilon  + \Delta{\chi}_r.
\end{equation}
The optimal gain for the total variate \eqref{eqn:optimallinearcv} is:
\begin{equation}
\label{eqn:projected-control-optimal-S}
 \*S^{\textnormal{opt}}  = \*\Phi_r +  \Cov{\Delta{\chi}_r}{\hat\upsilon}\,(\Cov{\hat\upsilon}{\hat\upsilon})^{-1}.
\end{equation}
Using the approximate gain matrix
\begin{equation}
\label{eqn:projected-control-S}
    \*S = \*\Phi_r,
\end{equation}
in \eqref{eqn:linearcv} removes the variability of $\chi$ within $\widehat{\mathcal{S}}_\upsilon$.
\end{theorem}
\begin{proof}
From \eqref{eqn:variates-in-Phi} we have:
\begin{equation*}
(\chi -\Mean{\chi})(\hat\upsilon - \Mean{\hat\upsilon})^T =  \*\Phi_r\,(\hat\upsilon - \Mean{\hat\upsilon})(\hat\upsilon - \Mean{\hat\upsilon})^T + (\Delta{\chi}_r - \Mean{\Delta{\chi}_r})(\hat\upsilon - \Mean{\hat\upsilon})^T.
\end{equation*}
Taking expected values and replacing in the optimal gain formula \eqref{eqn:optimallinearcv} 
%
%\begin{equation}
%\begin{split}
%\Cov{\hat{\upsilon}}{\hat{\upsilon}} &= \Cov{\*c_{1:r}}{\*c_{1:r}}, \quad
%\Cov{\chi}{\hat{\upsilon}} = \*\Phi\, \Cov{\*c_{1:n}}{\*c_{1:r}}, \\
% \*S^{\textnormal{opt}} &= \Cov{\chi}{\hat{\upsilon}}\,\Cov{\hat{\upsilon}}{\hat{\upsilon}}^{-1} =  \*\Phi\, \begin{bmatrix} \mathbf{I}_r \\ \Cov{\*c_{r+1:n}}{\*c_{1:r}}\,(\Cov{\*c_{1:r}}{\*c_{1:r}})^{-1}\end{bmatrix} \\
% &= \*\Phi_{:,1:r} +  \*\Phi_{:,r+1:n}\, \Cov{\*c_{r+1:n}}{\*c_{1:r}}\,(\Cov{\*c_{1:r}}{\*c_{1:r}})^{-1},
%%  \\
%%&= \*\Phi_r +  \Cov{\Delta{\chi}_r}{\hat\upsilon}\,(\Cov{\hat\upsilon}{\hat\upsilon})^{-1}.
% \end{split}   
%\end{equation}
%
gives \eqref{eqn:projected-control-optimal-S}.
From \eqref{eqn:linearcv}:
\begin{equation*}
\begin{split}
\zeta &= \chi - \*S\,(\hat\upsilon - \Mean{\hat\upsilon}) = \left(\mathbf{I} - \*S\,\*\Phi_r^* \right)\,\chi + \*S\,\Mean{\hat\upsilon}
%&=  \left(\mathbf{I} - \*S\,\*\Phi_r^* \right)\,
%( \*\Phi_r\, \*c_{1:r} + \*\Phi_{r+1:n}\, \*c_{r+1:n} ) + \*S\,\Mean{\hat\upsilon} \\
= \left(\*\Phi_r - \*S \right)\,\hat{\upsilon} +  \Delta{\chi}_r + \*S\,\Mean{\hat\upsilon}, 
%    \Cov{\zeta}{\zeta} 
%    &= \Cov{\chi}{\chi} + \*S\,\Cov{\hat\upsilon}{\hat\upsilon}\,\*S^\intercal -  \*S\,\Cov{\hat\upsilon}{\chi} - \Cov{\chi}{\hat\upsilon}\,\*S^\intercal \\
%    &\stackrel{\cref{eqn:projected-control-S}}{=} \Cov{\chi}{\chi} + \*S\,\*\Phi_r^*\,\Cov{\chi}{\chi}\,(\*\Phi_r^*)^\intercal\,\*S^\intercal -  \*S\,\*\Phi_r^*\,\Cov{\chi}{\chi} - \Cov{\chi}{\chi}\,(\*\Phi_r^*)^\intercal\,\*S^\intercal \\
%    &= \left(\mathbf{I} - \*S\,\*\Phi_r^* \right)\,\Cov{\chi}{\chi}\, \left(\mathbf{I} - \*S\,\*\Phi_r^* \right)^\intercal.
\end{split}   
\end{equation*}
and the approximate gain \eqref{eqn:projected-control-S} leads to $\Cov{\zeta}{\zeta}=\Cov{\Delta{\chi}_r}{\Delta{\chi}_r}$. Since $\Delta{\chi}_r = (\mathbf{I} -  \*\Phi_r\, \*\Phi_r^*) \, \chi$ is $\mathbf{M}_\upsilon$-orthogonal to $\widehat{\mathcal{S}}_\upsilon$ the variability of $\chi$ within $\widehat{\mathcal{S}}_\upsilon$ has been removed.
\end{proof}
A consequence of \cref{thm:optimalavggain} is that the approximate gain \eqref{eqn:projected-control-S} is constant in time.

\begin{remark}[Approximation of optimal gain]
\label{rem:approx-optimal-gain}
If the mean of the control variate is unknown, then an ancillary variate is used \cref{eqn:CV-independent-mean}. 
If the ancillary variate has a second moment that is equal to that of the control variate, then by~\cref{eqn:optimal-gain-half} the optimal gain is approximately $\*S \approx \*\Phi_r/2$.
\end{remark}

\begin{remark}[Gain error]
In \eqref{eqn:variates-in-Phi} the reduced order approximation error is:
\[
\Delta{\chi}_r = \chi - \*\Phi_r\,\hat\upsilon.
\]
The fixed gain \eqref{eqn:projected-control-S} is a good approximation of the optimal gain \eqref{eqn:projected-control-optimal-S} when the term $ \Cov{\Delta{\chi}_r}{\hat\upsilon}\,(\Cov{\hat\upsilon}{\hat\upsilon})^{-1}$ is small, i.e., when the covariance between the approximation error $\Delta{\chi}_r$ and the reduced order projection $\hat\upsilon$ is small relative to the covariance of the reduced order projection.
\end{remark}

\begin{remark}
Upper bounds for the error between the forecasted full model state, and forecasted reduced order model state that functions as a control variate, are available in the literature \cite{iliescu2014variational,KV01,singler2014new}.
%
\iffalse
For example, for the  Navier-Stokes equations~\cite{KV01,singler2014new}, under suitable assumptions, the ROM error estimate takes the form:
%
\begin{eqnarray}
    \norm{\chi - \*\Phi_r\hat\upsilon}_{L^2}^2 
    \leq C \, 
    \left(
        \left(1 + \norm{S_r} \right) \, 
        \left(
            h^{2 m}
            + \Delta t^{2 k}
        \right)
        + \sum \limits_{i=r}^{R}\lambda_i
        + \sum \limits_{i=r}^{R}\lambda_i \, |  \phi_i |_{H^1}
    \right) \, ,
    \label{eqn:rom-error-nse}
\end{eqnarray}
where
$h$ is the mesh-size and $m$ is the finite element order used in the FOM spatial discretization,
$\Delta t$ is the time step and $k$ is the order of the time discretization (e.g., $k=1$ for backward Euler), 
$r$ is the ROM dimension,
$R$ is the rank of the snapshot matrix,
$\phi_i$ is the  $i$-th ROM basis function and $\lambda_i$ its corresponding eigenvalue,
$| \cdot |_{H^1}$ is the  $H^1$ seminorm,
$S_r$ is the ROM stiffness matrix, 
and $C$ is a generic constant that can depend on the problem data, but not on $h, \Delta  t, \lambda_i$, or $\phi_i$.
\fi
%
Assume that the deviations from the mean of $\Delta{\chi}_r$ and $\hat\upsilon$ are bounded by a moderate constant times the respective means.
A simple scale analysis in \eqref{eqn:projected-control-optimal-S} shows that 
\[
\Vert \*S^{\textnormal{opt}}  - \*S \Vert = 
\Vert \Cov{\Delta{\chi}_r}{\hat\upsilon}\,(\Cov{\hat\upsilon}{\hat\upsilon})^{-1} \Vert \sim \frac{\Vert \Delta{\chi}_r \Vert}{\Vert  \hat\upsilon  \Vert} =
 \frac{\Vert \Delta{\chi}_r \Vert}{\Vert \chi - \Delta{\chi}_r \Vert} 
\le \frac{\Vert \Delta{\chi}_r \Vert/\Vert \chi \Vert}{1 - (\Vert \Delta{\chi}_r \Vert/\Vert \chi \Vert)},
\]
so the smaller the ROM error is, the closer the fixed approximate gain \eqref{eqn:projected-control-S} is to the optimal one \eqref{eqn:projected-control-optimal-S}.
\end{remark}

\begin{remark}
The discussion in this section applies with minor changes to the infinite dimensional case. Consider an infinite dimensional principal space $\mathcal{S}_\chi$ with an inner product $\inner{\cdot}{\cdot}_{\mathcal{S}_\chi}$, and a possibly infinite-dimensional control space  $\mathcal{S}_\upsilon$ with the inner product $\inner{\cdot}{\cdot}_{\mathcal{S}_\upsilon}$. Consider a second dot product $\inner{\cdot}{\cdot}_{\widehat{\mathcal{S}}_\upsilon}$ on $\mathcal{S}_\chi$ (motivated by the physics of the problem). A linear bounded operator $\*\Phi_r : \mathcal{S}_\upsilon \to \mathcal{S}_\chi$  links the control and primal spaces; let $\widehat{\mathcal{S}}_\upsilon = \textnormal{Range}\{ \*\Phi_r \}$.  The adjoint operator $\Phi_r^\ast:  \mathcal{S}_\chi  \to \mathcal{S}_\upsilon$ defined by $\langle \*\Phi_r^\ast \*x, \*u \rangle_{\mathcal{S}_\upsilon} = \langle \*x, \*\Phi_r \*u \rangle_{\widehat{\mathcal{S}}_\upsilon}$ gives the control variate relation \eqref{eqn:projected-control}. 
\end{remark}

%%%%%%%%%%%%%%%%%
%!TEX root = main.tex
%%%%%%%%%%%%%%%%%%%%%%%%%%%%%%%%%%%%
\section{Multifidelity EnKF with ROM control variates.}
\label{sec:twospacepodenkf}

We now build a multifidelity EnKF using the multivariate control variate framework, with the transitions between fidelities defined in terms of optimal gains, leading to the multifidelity approach, which is different than the MLEnKF idea discussed in \cite{Hoel_2016_MLEnKF}.

For ease of exposition, a two-fidelity variant of the MFEnKF with ROM control variates is discussed first, and a telescopic generalization to $\!L$ fidelities is presented later. The schematic working of a two levels of fidelity MFEnKF is illustrated in Figure \ref{fig:MFEnKF}.
\begin{figure}[t]
\centering
  \begin{tikzpicture}

\node[draw=none, text=red, align=center] at (0.75,2) {Full order\\ model\\ space};
\node[draw=none, text=red, align=center] at (2.5,2.5) {$X^{\|a}_{i-1}$};
\node[draw=black, fill=red, circle, minimum size=1.0em](Xa1) at (2,2){};
\node[draw=black, fill=red, circle, minimum size=1.0em](Xa2) at (3,2){};

\node[draw=none, text=RoyalBlue , align=center] at (0.75,-0.75) {Reduced\\ order\\ model\\ space};
\node[draw=none, text=RoyalBlue, align=center] at (2.5,0) {$\hat{U}^{\|a}_{i-1}$};
\node[draw=black, fill=RoyalBlue , circle, minimum size=0.5em, inner sep=0em](Uha1) at (2,0.5){};
\node[draw=black, fill=RoyalBlue , circle, minimum size=0.5em, inner sep=0em](Uha2) at (3,0.5){};

\draw [->, thick] (Xa1) -- (Uha1);
\draw [->, thick] (Xa2) -- (Uha2);
\node at (2.5, 1.25) {$\*\Phi^*_r$};

\node[draw=black, fill=RoyalBlue , circle, minimum size=0.5em, inner sep=0em](Ua11) at (2,-1){};
\node[draw=black, fill=RoyalBlue , circle, minimum size=0.5em, inner sep=0em](Ua12) at (2.333,-1){};
\node[draw=black, fill=RoyalBlue , circle, minimum size=0.5em, inner sep=0em](Ua13) at (2.6666,-1){};
\node[draw=black, fill=RoyalBlue , circle, minimum size=0.5em, inner sep=0em](Ua14) at (3,-1){};

\node[draw=black, fill=RoyalBlue , circle, minimum size=0.5em, inner sep=0em](Ua21) at (2,-1.333){};
\node[draw=black, fill=RoyalBlue , circle, minimum size=0.5em, inner sep=0em](Ua22) at (2.333,-1.333){};
\node[draw=black, fill=RoyalBlue , circle, minimum size=0.5em, inner sep=0em](Ua23) at (2.6666,-1.333){};
\node[draw=black, fill=RoyalBlue , circle, minimum size=0.5em, inner sep=0em](Ua24) at (3,-1.333){};

\node[draw=black, fill=RoyalBlue , circle, minimum size=0.5em, inner sep=0em](Ua31) at (2,-1.666){};
\node[draw=black, fill=RoyalBlue , circle, minimum size=0.5em, inner sep=0em](Ua32) at (2.333,-1.666){};
\node[draw=black, fill=RoyalBlue , circle, minimum size=0.5em, inner sep=0em](Ua33) at (2.6666,-1.666){};
\node[draw=black, fill=RoyalBlue , circle, minimum size=0.5em, inner sep=0em](Ua34) at (3,-1.666){};

\node[draw=none, text=RoyalBlue, align=center] at (2.5,-2.333) {$U^{\|a}_{i-1}$};

%%draw rectangle
\draw[dashed, very thick, fill=gray!10] ($(5,-2.666) - (1em, 0)$) rectangle ($(6,2.9) + (1em, 0)$);

%% Forecast ensembles
\node[draw=none, text=red, align=center] at (5.5,2.5) {$X^{\|f}_{i}$};
\node[draw=black, fill=red, circle, minimum size=1.0em](Xf1) at (5,2){};
\node[draw=black, fill=red, circle, minimum size=1.0em](Xf2) at (6,2){};

\node[draw=none, text=RoyalBlue, align=center] at (5.5,0) {$\hat{U}^{\|f}_{i}$};
\node[draw=black, fill=RoyalBlue , circle, minimum size=0.5em, inner sep=0em](Uhf1) at (5,0.5){};
\node[draw=black, fill=RoyalBlue , circle, minimum size=0.5em, inner sep=0em](Uhf2) at (6,0.5){};

\node[draw=black, fill=RoyalBlue , circle, minimum size=0.5em, inner sep=0em](Uf11) at (5,-1){};
\node[draw=black, fill=RoyalBlue , circle, minimum size=0.5em, inner sep=0em](Uf12) at (5.333,-1){};
\node[draw=black, fill=RoyalBlue , circle, minimum size=0.5em, inner sep=0em](Uf13) at (5.6666,-1){};
\node[draw=black, fill=RoyalBlue , circle, minimum size=0.5em, inner sep=0em](Uf14) at (6,-1){};

\node[draw=black, fill=RoyalBlue , circle, minimum size=0.5em, inner sep=0em](Uf21) at (5,-1.333){};
\node[draw=black, fill=RoyalBlue , circle, minimum size=0.5em, inner sep=0em](Uf22) at (5.333,-1.333){};
\node[draw=black, fill=RoyalBlue , circle, minimum size=0.5em, inner sep=0em](Uf23) at (5.6666,-1.333){};
\node[draw=black, fill=RoyalBlue , circle, minimum size=0.5em, inner sep=0em](Uf24) at (6,-1.333){};

\node[draw=black, fill=RoyalBlue , circle, minimum size=0.5em, inner sep=0em](Uf31) at (5,-1.666){};
\node[draw=black, fill=RoyalBlue , circle, minimum size=0.5em, inner sep=0em](Uf32) at (5.333,-1.666){};
\node[draw=black, fill=RoyalBlue , circle, minimum size=0.5em, inner sep=0em](Uf33) at (5.6666,-1.666){};
\node[draw=black, fill=RoyalBlue , circle, minimum size=0.5em, inner sep=0em](Uf34) at (6,-1.666){};
\node[draw=none, text=RoyalBlue, align=center] at (5.5,-2.333) {$U^{\|f}_{i}$};

%% Draw arrows for next stage
\node[draw=none,align=center] at (3.9,3.5) {Forecast\\ step};
\draw [->, ultra thick, draw=red] ($(Xa2) + (1.25em, 0)$) -- ($(Xf1) - (1.25em, 0)$) node[midway,above=0.333em, text=red] {$\!M^X_{i-1,i}$};
\draw [->, thick, draw=RoyalBlue] ($(Uha2) + (1.25em, 0)$) -- ($(Uhf1) - (1.25em, 0)$) node[midway,above=0.333em, text=RoyalBlue] {$\!M^U_{i-1,i}$};
\draw [->, thick, draw=RoyalBlue] ($(Ua24) + (1.25em, 0)$) -- ($(Uf21) - (1.25em, 0)$) node[midway,above=0.333em, text=RoyalBlue] {$\!M^U_{i-1,i}$};

%% analysis ensembles
\node[draw=none, text=red, align=center] at (8.5,2.5) {$X^{\|a}_{i}$};
\node[draw=black, fill=red, circle, minimum size=1.0em](Xaa1) at (8,2){};
\node[draw=black, fill=red, circle, minimum size=1.0em](Xaa2) at (9,2){};

%\node[draw=none, text=RoyalBlue, align=center] at (8.5,0) {$\hat{U}^{\|a}_{i}$};
%\node[draw=black, fill=RoyalBlue , circle, minimum size=0.5em, inner sep=0em](Uhaa1) at (8,0.5){};
%\node[draw=black, fill=RoyalBlue , circle, minimum size=0.5em, inner sep=0em](Uhaa2) at (9,0.5){};

\node[draw=black, fill=RoyalBlue , circle, minimum size=0.5em, inner sep=0em](Uaa11) at (8,-1){};
\node[draw=black, fill=RoyalBlue , circle, minimum size=0.5em, inner sep=0em](Uaa12) at (8.333,-1){};
\node[draw=black, fill=RoyalBlue , circle, minimum size=0.5em, inner sep=0em](Uaa13) at (8.6666,-1){};
\node[draw=black, fill=RoyalBlue , circle, minimum size=0.5em, inner sep=0em](Uaa14) at (9,-1){};

\node[draw=black, fill=RoyalBlue , circle, minimum size=0.5em, inner sep=0em](Uaa21) at (8,-1.333){};
\node[draw=black, fill=RoyalBlue , circle, minimum size=0.5em, inner sep=0em](Uaa22) at (8.333,-1.333){};
\node[draw=black, fill=RoyalBlue , circle, minimum size=0.5em, inner sep=0em](Uaa23) at (8.6666,-1.333){};
\node[draw=black, fill=RoyalBlue , circle, minimum size=0.5em, inner sep=0em](Uaa24) at (9,-1.333){};

\node[draw=black, fill=RoyalBlue , circle, minimum size=0.5em, inner sep=0em](Uaa31) at (8,-1.666){};
\node[draw=black, fill=RoyalBlue , circle, minimum size=0.5em, inner sep=0em](Uaa32) at (8.333,-1.666){};
\node[draw=black, fill=RoyalBlue , circle, minimum size=0.5em, inner sep=0em](Uaa33) at (8.6666,-1.666){};
\node[draw=black, fill=RoyalBlue , circle, minimum size=0.5em, inner sep=0em](Uaa34) at (9,-1.666){};
\node[draw=none, text=RoyalBlue, align=center] at (8.5,-2.333) {$U^{\|a}_{i}$};

%% arrows

\node[draw=none,align=center] at (7.1,3.5) {Analysis\\ step};
\draw [->, thick] ($(Xf2) + (1.25em, 0)$) -- ($(Xaa1) - (1.25em, 0)$) node[midway, below=0.333em, name=KX, text=red] {$\widetilde{\*K}_i$};
\draw [->, thick] ($(Uf24) + (1.25em, 0)$) -- ($(Uaa21) - (1.25em, 0)$) node[midway, above=0.333em, name=KU, text=RoyalBlue] {$\*\Phi^*_r\widetilde{\*K}_i$};

\draw[rounded corners=1em, ->, thick]($(Uhf2) + (1.25em, 0)$) -| (KX);
\draw[rounded corners=1em, ->, thick]($(Uhf2) + (1.25em, 0)$) -| (KU);

\end{tikzpicture}
  \caption{Conceptual working of a two-fidelity MFEnKF with a ROM control variate.}
  \label{fig:MFEnKF}
\end{figure}
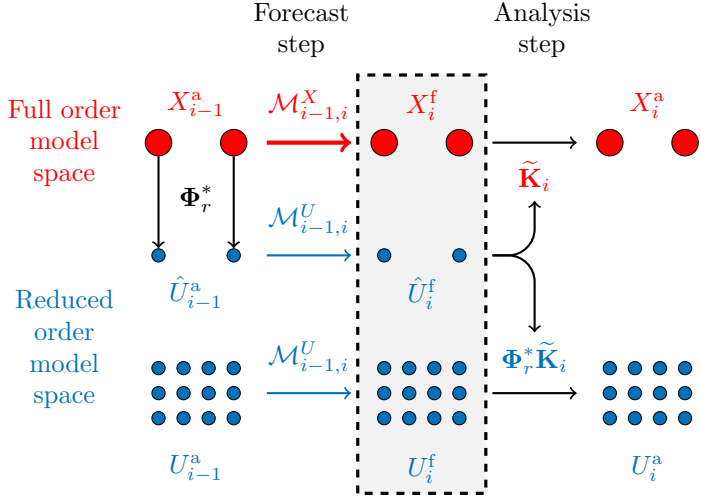

\begin{assumption}[Setting for constructing MFEnKF]
\begin{enumerate}
\item Two numerical models \eqref{eqn:umodel} of the same natural phenomenon are available. The first one is the full order model, that propagates a state $\x \in \Re^n$ in the full order space via the FOM dynamics $\Model^{\x}$. The second one is a reduced order model, that propagates a reduced order state $\u \in \Re^r$ via the ROM dynamics $\Model^{\u}$. The distribution of $\x$ embodies our knowledge about the state represented in the FOM space, and the distribution of $\u$ our knowledge about the state (represented in the ROM space).
\item Projection operators \eqref{eqn:projection-operators} are available, that map the full space onto the reduced one, $U = \*\Phi_r^* X$, and the reduced space into the full one, $X = \*\Phi_r U$, such that $\*\Phi_r^* \*\Phi_r = \mathbf{I}_r$.
\item A full space observation operator $\Hobs_i$ maps the FOM state space onto the observation space. A reduced space observation operator $\Hobs_{r,i}$ maps the ROM state space onto the observation space. The observation errors have covariances $\Cov{\erro_i}{\erro_i}$ and $\Cov{\erro_i^{\u}}{\erro_i^{\u}}$, respectively. The reduced space observation operator is assumed to be consistent with the full space observation operator, in the sense that:
\begin{align*}
    \Hobs_{r,i}(\ub_i) &\approx \Hobs_i(\boldsymbol{\Phi}_r \ub_i)
    \quad \Rightarrow \quad \Hobs'_{r,i}|_{\ub_i} \approx \Hobs_i'|_{\boldsymbol{\Phi}_r \ub_i}\, \boldsymbol{\Phi}_r.
    \end{align*}
\end{enumerate}
\end{assumption}

Our goal is to build an ensemble Kalman filter that takes advantage of two models, and can leverage the higher accuracy of the FOM and the lower cost of the ROM by using small ensembles of FOM runs in conjunction with large ensembles of ROM runs. A first possible approach is to use EnKF in the FOM space and employ multimodel ensembles to build empirical covariances. However, building empirical moments from ensemble members of different sizes is challenging. For example, one needs to project the ROM ensemble members into the full space and carry out the inference there.  A second possible approach is to ``stack'' the two models, and obtain a supermodel that advances the combined ROM and FOM states; ``stack'' the observation operators; and apply EnKF in the aggregated state space. This second approach, however, cannot employ different numbers of FOM and ROM ensemble members, and is likely to suffer when the ROM and FOM solutions are poorly correlated. A third approach is multilevel EnKF (MLEnKF) \cite{Hoel_2016_MLEnKF,Chernov_2017_MLEnKF}, where EnKF is applied in the FOM space, and ROM runs are (only) used to improve the empirical covariance estimates.

MFEnKF takes advantage of the availability of two models by employing a control variate framework. 
The FOM state $\x \equiv \chi$ is the principal variate. The ROM state $\uh \equiv \hat\upsilon$ is the control variate, and another ROM state $\u \equiv \upsilon$ its ancillary variate. We focus on projection control variates $\uh = \*\Phi_r^* \x$. The  total variates $\zb$ and $\za$ represent our combined prior and posterior knowledge, respectively, at time $t_i$ through the linear control variate technique~\cref{eqn:linearcv}:
\begin{equation}
\label{eqn:z-controlvar}
%    \zb_i = \xb_i - \frac{1}{2}\,\*\Phi_r\,(\uhb_i - \ub_i),
    \zb_i = \xb_i - \mathbf{S}_i\,(\uhb_i - \ub_i), \qquad
    \za_i = \xa_i - \mathbf{S}_i\,(\uha_i - \ua_i).
\end{equation}
The main idea of MFEnKF is to apply EnKF \cref{eqn:enkfanalysis} to the total variate \eqref{eqn:z-controlvar}, under the following restrictions.
\begin{assumption}[Restrictions in constructing MFEnKF]
\begin{enumerate}
\item One can run the FOM and the ROM, but there is no dynamical model associated with the total variate \eqref{eqn:z-controlvar}.
Consequently, one does not sample $Z$ directly. Rather, the uncertainty in the total variate is represented by the three ensembles of principal, control, and ancillary variates:
\begin{equation}
\label{eqn:three-ensembles}
    \zb \colonapprox \left(\En{\xb}, \En{\uh^\|b}, \En{\ub}\right), \quad 
    \za \colonapprox \left(\En{\xa}, \En{\uh^\|a}, \En{\ua}\right).
\end{equation}
The MFEnKF forecast step propagates the three ensembles forward in time to obtain a representation of the prior total variate, and the analysis step produces three ensembles representing the posterior total variate. 
\item One can observe the principal variate using the full space observation operator $\Hobs_i$, and the control and ancillary variates using the reduced space observation operator $\Hobs_{r,i}(\ub_i)$. However, one does not observe the total variate $\zb_i$ directly.  Instead, we consider the following indirect observation operator:
\begin{equation}
\label{eqn:hofz-approximate}
    \overline{\Hobs}_i(\zb_i) = \Hobs_i(\xb_i) - \mathbf{T}_i\, \left(\Hobs_{r,i}(\uhb_i) - \Hobs_{r,i}(\ub_i) \right).
\end{equation}
%
%Note that by defining $\mathbf{T}_i = \Hobs_i'(\xb_i)\,\mathbf{S}_i$, $\Hobs_{r,i} = \Id_{\nobs}$, one obtains a first order approximation to the observation operator applied to the total variate \cref{eqn:z-controlvar}:
%%
%\begin{equation}
%\label{eqn:hofz-controlvar}
%    \Hobs_i(\zb_i) \approx \overline{\Hobs}_i(\zb_i) = \Hobs_i(\xb_i) - \Hobs_i'(\xb_i)\,\mathbf{S}_i\,\uhb_i + \Hobs_i'(\xb_i)\,\mathbf{S}_i\,\ub_i,
%\end{equation}
%%
%as can be immediately established by Taylor expansion.
%
%
We are interested in indirect observations \eqref{eqn:hofz-approximate} that approximate, to first order, the nonlinear observation operator $\Hobs_i$ applied to the total variate \cref{eqn:z-controlvar}:
\begin{subequations}
\begin{equation}
\label{eqn:hofz-approximate-H}
\begin{split}
    \Hobs_i(\zb_i) &= \Hobs_i(\xb_i) - \Hobs_i'|_{\xb_i}\, \mathbf{S}_i\, (\uhb_i - \ub_i)  +\text{h.o.t.},\\
    \overline{\Hobs}_i(\zb_i) &= \Hobs_i(\xb_i) - \mathbf{T}_i\, \Hobs'_{r,i}|_{\boldsymbol{\Phi}_r^\ast\,\xb_i}(\uhb_i - \ub_i)+\text{h.o.t.}
\end{split}    
\end{equation}
This is achieved by choosing a matrix $\mathbf{T}_i$ such that
\begin{equation}
\label{eqn:hofz-approximate-T}
    \Hobs_i'|_{\xb_i}\, \mathbf{S}_i  \approx  \mathbf{T}_i\, \Hobs'_{r,i}|_{\boldsymbol{\Phi}_r^\ast\xb_i}
    \approx  \mathbf{T}_i\, \Hobs'_{i}|_{\boldsymbol{\Phi}_r \boldsymbol{\Phi}_r^\ast\xb_i}\, \boldsymbol{\Phi}_r.
\end{equation}
\end{subequations}
\end{enumerate}
\end{assumption}

\subsection{Forecast step.}
%%%%%%%%%%%%%%%%%%%%%%%%%

In order to ensure that the analysis control variate $\uha_{i-1}$  is highly correlated with the corresponding principal variate $\xa_{i-1}$, $\uha_{i-1}$ is not obtained through EnKF analysis  \cref{eqn:enkfanalysis}; rather, it is obtained by projecting the principal variate (the FOM analysis state) onto the reduced space:
\begin{subequations}
\label{eqn:MFEnKF-forecast-step}
\begin{equation}
\label{eqn:MFEnKF-forecast-step-projection}
\uha_{i-1} \coloneqq \*\Phi_r^*\,\xa_{i-1}.
\end{equation}
The MFEnKF forecast step propagates each of the three analysis ensembles \eqref{eqn:three-ensembles} at time $t_{i-1}$ forward to time $t_{i}$:
\begin{equation}
\label{eqn:MFEnKF-forecast-step-model}
\begin{split}
    \xb[k]_i  &= \Model^{\x}_{i-1,i}(\xa[k]_{i-1}),\qquad
    \uhb[k]_i = \Model^{\u}_{i-1,i}(\uha[k]_{i-1}),\qquad k=1,\dots,N_X; \\
    \ub[k]_i &= \Model^{\u}_{i-1,i}(\ua[k]_{i-1}),\qquad k=1,\dots,N_U.
    \end{split}
\end{equation}
\end{subequations}
The MFEnKF forecast step \eqref{eqn:MFEnKF-forecast-step} is illustrated in Figure \ref{fig:MFEnKF}.
% We conjecture that this correction is vitally important to the stable functioning of the algorithm.

\begin{remark}[Assumption of independence] 
\label{rem:assume-independence}
In the control variate framework the control $\uhb_i$ and ancillary $\ub_i$ variates are independent random variables. In MFEnKF the analysis step will correlate the principal $\xb_i$ and the ancillary variates. Nevertheless, using typical statistical Kalman gain independence assumptions in the EnKF, we will treat $\uhb_i$ and $\ub_i$ as independent in MFEnKF calculations.
\end{remark}

\begin{remark}[Forecast step and ROM bias] 
For linear models %$\Model^{\u}_{i-1,i} = \*\Phi_r^*\,\Model^{\x}_{i-1,i}\,\*\Phi_r$ 
another control variate highly correlated with the principal variate can be obtained by direct projection
\begin{equation}
\label{eqn:FOM-control-var}
    \uhb_i = \*\Phi_r^*\,\xb_i,
\end{equation}
saving the additional ROM runs for $\uhb_i$ required by \eqref{eqn:MFEnKF-forecast-step-model}.
In general, however, ROMs are affected by systematic bias
\begin{equation*}
    \Model^{\u}_{i-1,i}( \*\Phi_r^*\,\x_{i-1}) - \*\Phi_r^*\,\Model^{\x}_{i-1,i}(\x_{i-1}) = \beta_i.
\end{equation*}
While the ancillary variate $\ub_i$ computed using the ROM  \eqref{eqn:MFEnKF-forecast-step-model} is affected by this bias, the control variate \eqref{eqn:FOM-control-var} obtained by direct projection is not, and in general \eqref{eqn:FOM-control-var} violates the underlying probabilistic assumptions~\cref{eqn:z-controlvar}.

In contrast, the forecasting strategy \eqref{eqn:MFEnKF-forecast-step} computes both the control variate $\uhb_i$ as well as the ancillary variate $\ub_i$ as solutions of the same ROM model \eqref{eqn:MFEnKF-forecast-step-model}. Consequently, they are both affected by the ROM bias. If $\beta_i$ is independent of the ROM state then the biases in control and ancillary variates cancel each other out in \eqref{eqn:z-controlvar};  if the bias is not constant this strategy is still likely to significantly reduce it.
\end{remark}

%%%%%%%%%%%%%%%%%%%%%%%%%
\subsection{Analysis step.}
%%%%%%%%%%%%%%%%%%%%%%%%%
%

We focus on the case where the control variate is $\uhb_i = \boldsymbol{\Phi}_r^\ast\,\xb_i$, the reduced observation operator is $\Hobs_{r,i}(\ub_i) = \Hobs_i(\*\Phi_r\,\ub_i)$, and the gain matrix is $\mathbf{S}_i = \*\Phi_r/2$, per \cref{rem:approx-optimal-gain}. In this case equation \eqref{eqn:hofz-approximate-T} is satisfied exactly by $\mathbf{T}_i = 1/2$, and the indirect observation operator \eqref{eqn:hofz-approximate} reads:
\begin{equation}
\label{eqn:hofz-controlvar-phi}
    \overline{\Hobs}_i(\zb_i) = \Hobs_i(\xb_i) - \frac{1}{2}\Hobs_{r,i}(\uhb_i) + \frac{1}{2}\Hobs_{r,i}(\ub_i).
\end{equation}
%
%\begin{proof}
%By Taylor expansion.
%\begin{align*}
%    \Hobs_i(\zb_i) &= \Hobs_i(\xb_i) - \frac{1}{2} \Hobs_i'|_{\xb_i}\puhb_i + \frac{1}{2} \Hobs_i'|_{\xb_i}\pub_i +\text{h.o.t.},\\
%    \Hobs_i(\pub_i) &= \Hobs_i(\xb_i) + \Hobs_i'|_{\xb_i}\pub_i +\text{h.o.t.},\\
%    \Hobs_i(\puhb_i) &= \Hobs_i(\xb_i) + \Hobs_i'|_{\xb_i}\puhb_i +\text{h.o.t.}.
%\end{align*}
%\end{proof}
%
Using \cref{rem:assume-independence} we have that:
\begin{align*}
\MeanE{\overline{\Hobs}_i(\zb_i)} &= \MeanE{\Hobs(\xb_{i})} - \frac{1}{2}\,\MeanE{\Hobs_i(\*\Phi_r\,\uhb_i)} + \frac{1}{2}\,\MeanE{\Hobs_i(\*\Phi_r\ub_i)}, \\
\CovE{\overline{\Hobs}_i(\zb_i)}{\overline{\Hobs}_i(\zb_i)} &= \CovE{\Hobs_i(\xb_{i})}{\Hobs_i(\xb_{i})} 
+ \frac{1}{4}\CovE{\Hobs_i(\*\Phi_r\uhb_i)}{\Hobs_i(\*\Phi_r\uhb_i)} \\
& - \frac{1}{2}\CovE{\Hobs_i(\xb_{i})}{\Hobs_i(\*\Phi_r\uhb_i)}
- \frac{1}{2}\CovE{\Hobs_i(\*\Phi_r\uhb_i)}{\Hobs_i(\xb_{i})}
+ \frac{1}{4}\CovE{\Hobs_i(\*\Phi_r\ub_i)}{\Hobs_i(\*\Phi_r\ub_i)}.
\end{align*}
The covariance $\CovE{\zb_i}{\overline{\Hobs}_i(\zb_i)}$ is defined in a similar manner, and the empirical Kalman gain for the total variate is computed as follows:
\begin{equation}
\label{eqn:MFEnKF-Kalman-gain}
\widetilde{\K}_i \coloneqq  \CovE{\zb_i}{\overline{\Hobs}_i(\zb_i)}\, \Bigl(\CovE{\overline{\Hobs}_i(\zb_i)}{\overline{\Hobs}_i(\zb_i)} + \Cov{\erro^Z_i}{\erro^Z_i}\Bigr)^{-1}.
\end{equation}
The perturbed observations EnKF \eqref{eqn:enkfanalysis} is applied  using the indirect observations \cref{eqn:hofz-controlvar-phi} to estimate the total variate \cref{eqn:z-controlvar}: 
\begin{equation}
\label{eqn:kf-total-variate}
    \za_i = \zb_i - \widetilde{\mathbf{K}}_i\, \left(\overline{\Hobs}_i(\zb_i) - \Obs_i - \erro^{Z}_i\right),
\end{equation}
where $\erro^{Z}_i$ is an independent variable that represents the perturbations to be added to the indirect observations.
We make the ansatz:
\begin{equation}
\label{eqn:observation-noise}
    \erro^{Z}_i = \erro^{\x}_i - \frac{1}{2}\erro^{\uh}_i + \frac{1}{2}\erro^{\u}_i,
\end{equation}
such that we have:
\begin{equation*}
\begin{split}
    \overline{\Hobs}_i(\zb_i) - \erro^{Z}_i 
    %= \Hobs_i(\xb_{i}) - \frac{1}{2}\Hobs_i(\*\Phi_r \uhb_i) + \frac{1}{2}\Hobs_i(\*\Phi_r\\ub_i) + \erro_i \\
    &= \left(\Hobs_i(\xb_{i}) - \erro^{\x}_i\right) - \frac{1}{2}\left(\Hobs_i(\*\Phi_r \uhb_i) - \erro^{\uh}_i\right) + \frac{1}{2}\left(\Hobs_i(\*\Phi_r \ub_i) - \erro^\u_i\right).
    \end{split}
\end{equation*}
The MFEnKF analysis step transforms the three background ensembles \eqref{eqn:three-ensembles} into three posterior ensembles. Using the EnKF update \eqref{eqn:kf-total-variate},  the representation of the total variates \eqref{eqn:z-controlvar}, and the representation of the observation error \eqref{eqn:observation-noise}, we have:
\begin{equation*}
\begin{split}
& \underbrace{\xa_i - \frac{1}{2}\,\*\Phi_r\,(\uha_i - \ua_i)}_{\za_i} = \underbrace{\xb_i - \frac{1}{2}\,\*\Phi_r\,(\uhb_i - \ub_i)}_{\zb_i}  \\
& 
- \widetilde{\mathbf{K}}_i\,\Bigl(\underbrace{(\Hobs_i(\xb_i)  - \Obs_i + \erro^{\x}_i) - \frac{1}{2}\,(\Hobs_{r,i}(\uhb_i)  - \Obs_i - \erro^{\uh}_i) + \frac{1}{2}\,(\Hobs_{r,i}(\ub_i) - \Obs_i - \erro^\u_i)}_{\overline{\Hobs}_i(\zb_i) - \Obs_i - \erro_i} \Bigr).
\end{split}
\end{equation*}
Under the assumption that the all the information of $\zb$ in the orthogonal complement control space $\widehat{\!S}_{\u}^\perp$ does not effect the analysis control and ancillary variates, the MFEnKF transforms each of the variables \eqref{eqn:three-ensembles} as follows:
\begin{equation*}
\begin{split}
\xa_i &=\xb_i - \widetilde{\mathbf{K}}_i\,(\Hobs_i(\xb_i)  - \Obs_i + \erro^{\x}_i), \\
\uha_i &=\uhb_i - \*\Phi_r^*\,\widetilde{\mathbf{K}}_i\,(\Hobs_{r,i}(\uhb_i)  - \Obs_i + \erro^{\uh}_i),\\
\ua_i&= \ub_i - \*\Phi_r^*\,\widetilde{\mathbf{K}}_i\,(\Hobs_{r,i}(\ub_i) - \Obs_i + \erro^\u_i).
\end{split}
\end{equation*}
The background and analysis means of the total variate \cref{eqn:z-controlvar} are, respectively:
\begin{equation}
\label{eqn:Z-analysismean}
   \MeanE{\zb_i} = \MeanE{\xb_i} - \frac{1}{2}\*\Phi_r\,(\MeanE{\uhb_i} - \MeanE{\ub_i}),\quad
    \MeanE{\za_i} = \MeanE{\zb_i} - \widetilde{\mathbf{K}}_i\,(\MeanE{\overline{\Hobs}_i(\zb_i)} - \Obs_i).
\end{equation}
The MFEnKF analysis step \eqref{eqn:MFEnKF-forecast-step} is illustrated in Figure \ref{fig:MFEnKF}.

We consider two interpretations of the error in the indirect observations, which lead to different distributions of observation perturbations. 
Approach (i), called `total variate uncertainty consistency', interprets inference as occurring only on $\zb_i$, $\Hobs_i(\zb_i)$, and $Y$, with all other variates being a means to an end.
Approach (ii), called `control space uncertainty consistency', interprets the total variate as a means to an end, and focuses on the inference on primary and ancillary variates.

We first discuss approach (i), the `total variate uncertainty consistency'.
We require that $\overline{\Hobs}_i(\cdot) \approx \Hobs_i(\cdot)$ \eqref{eqn:hofz-approximate-H}, and that both operators have the same distribution of the observation errors, 
\begin{equation}
\label{eqn:same-obs-errors}
\Cov{\erro^Z_i}{\erro^Z_i} = \Cov{\erro^X_i}{\erro^X_i} = \Cov{\erro_i}{\erro_i}.
\end{equation}
To maintain the independence of the ancillary variate of both the principal and control variates, we make the natural assumption that $\erro^{\u}_i$ is independent of $\erro^{\x}_i$ and $\erro^{\uh}_i$. Consequently:
\begin{equation}
\label{eqn:etafull}
    \Cov{\erro^Z_i}{\erro^Z_i} = 
    \Cov{\erro^{\x}_i}{\erro^{\x}_i} + \frac{1}{4}\Cov{\erro^{\uh}_i}{\erro^{\uh}_i} +  \frac{1}{4}\Cov{\erro^{\u}_i}{\erro^{\u}_i} -  \frac{1}{2}\Cov{\erro^{\uh}_i}{\erro^{\x}_i} -  \frac{1}{2}\Cov{\erro^{\x}_i}{\erro^{\uh}_i}.
\end{equation}
%
%
%Assuming all other terms cancel (which is not true in general), from \cref{eqn:etafull} we can arrive at the following relationship,
%%
%\begin{equation}
%    \Cov{\eta^X}{\eta^X} + \frac{1}{4}\Cov{\eta^U}{\eta^U} + \frac{1}{4}\Cov{\eta^{\hat{U}}}{\eta^{\hat{U}}} - \frac{1}{2}\Cov{\eta^X}{\eta^{\hat{U}}} - \frac{1}{2}\Cov{\eta^{\hat{U}}}{\eta^X}  = \frac{1}{4}\Cov{\eta^X}{\eta^X} + \frac{1}{4}\Cov{\eta^U}{\eta^U},
%\end{equation}
%
%Using \cref{eqn:etafull},  \cref {eqn:total-variate-consistency} reduces to:
%%
%\begin{equation}
%    \frac{3}{4}\Cov{\eta^X}{\eta^X} + \frac{1}{4}\Cov{\eta^{\hat{U}}}{\eta^{\hat{U}}} - \frac{1}{2}\Cov{\eta^X}{\eta^{\hat{U}}} - \frac{1}{2}\Cov{\eta^{\hat{U}}}{\eta^X} = 0.
%\end{equation}
%%
To support the projection assumption \cref{eqn:MFEnKF-forecast-step-projection} we select $\eta^{\hat{U}} = \eta^X$. From \eqref{eqn:etafull} and \cref{eqn:same-obs-errors} we infer that $\Cov{\erro^{\u}_i}{\erro^{\u}_i} = 3\Cov{\erro_i}{\erro_i}$, and therefore
\begin{equation}
    \erro^{\x}_i = \erro^{\uh}_i \sim \!N(\*0, \Cov{\erro_i}{\erro_i}),\quad \erro^{\u}_i \sim \!N(\*0, 3\Cov{\erro_i}{\erro_i}).
\end{equation}

Note that replacing the analysis control variate with the projection of the analysis principal variate \cref{eqn:MFEnKF-forecast-step-projection} leads to a second possible definition of the analysis total variate: 
\begin{equation*}
\begin{split}
    \za_i = \xa_i - \frac{1}{2}\*\Phi_r(\uha_i - \ua_i),\quad
    \widetilde{Z}^\|a_i = \xa_i - \frac{1}{2}\*\Phi_r(\*\Phi^*_r \xa_i - \ua_i).
    \end{split}
\end{equation*}
The choice of observation  perturbations $\erro^{\x}$, $\erro^{\uh}_i$, and $\erro^{\u}$ in method (i) ensures the `total variate uncertainty consistency':
\begin{equation}
\label{eqn:total-variate-consistency}
    \Cov{\za_i}{\za_i} = \Cov{\widetilde{Z}^\|a_i}{\widetilde{Z}^\|a_i}.
\end{equation}
However in this view the inference on the ancillary variate has no direct physical meaning, and the assumed ancillary observation error is differs from the one used to construct the Kalman gain.

We now discuss approach (ii).  In this view the total variate is a means to an end. One runs multiple EnKFs for $\xa_i$, $\uha_i$, and $\ua_i$. Observations are taken in the full-order space, and in the reduced order space. The observations of control and ancillary variates use the same operator $\Hobs_{r,i}(\cdot)$, and therefore the errors have the same covariance $\Cov{\erro^U}{\erro^U}$. The observation errors are assumed to be:
\begin{equation*}
%\begin{split}
\eta^{\x}_i \sim\!N(\*0, \Cov{\erro_i}{\erro_i}), \quad  \eta^\u_i \sim\!N(\*0, \Cov{\erro^U_i}{\erro^U_i}),
\quad \eta^{\uh}_i = \Cov{\erro^U_i}{\erro^U_i}^{1/2}\,\Cov{\erro_i}{\erro_i}^{-1/2}\,\eta^{\x}_i \sim\!N(\*0, \Cov{\erro^U_i}{\erro^U_i}),
%\end{split}
\end{equation*}
such that $\eta^{\uh}_i $ and $\eta^{\x}_i$ are highly correlated, but $\eta^\u_i$ and $\eta^{\x}_i$ are independent. From \cref{eqn:etafull} the covariance of the total variate observation error is:
\begin{equation*}
\begin{split}
\Cov{\erro^Z_i}{\erro^Z_i} 
%&= \left( \mathbf{I} - \frac{1}{2}\,\Cov{\erro^U_i}{\erro^U_i}^{1/2}\,\Cov{\erro_i}{\erro_i}^{-1/2}\right)\,\Cov{\erro_i}{\erro_i}\, \left( \mathbf{I}- \frac{1}{2}\,\Cov{\erro_i}{\erro_i}^{-1/2}\Cov{\erro^U_i}{\erro^U_i}^{1/2}\right) + \frac{1}{4}\,\Cov{\erro^U_i}{\erro^U_i}\\
&= \Cov{\erro_i}{\erro_i} + \frac{1}{2}\,\Cov{\erro^U_i}{\erro^U_i}  - \frac{1}{2}\,\Cov{\erro^U_i}{\erro^U_i}^{1/2}\,\Cov{\erro_i}{\erro_i}^{1/2} - \frac{1}{2}\,\Cov{\erro_i}{\erro_i}^{1/2}\,\Cov{\erro^U_i}{\erro^U_i}^{1/2}.
\end{split}
\end{equation*}
If the errors of the reduced space observations are specified, then the above formula can be used to construct the empirical Kalman gain. For simplicity we consider in this paper that $\Cov{\erro^U_i}{\erro^U_i} = s_i^2\,\Cov{\erro_i}{\erro_i}$  such that $\eta^{\uh}_i = s_i\,\eta^{\x}_i$  and $\Cov{\erro^Z_i}{\erro^Z_i} = (1-s_i+s_i^2/2)\,\Cov{\erro_i}{\erro_i}$. 
If $s_i = 1$, then then we have the nice property that
\begin{equation}
    \Cov{\erro_i}{\erro_i} = \Cov{\erro^{\x}_i}{\erro^{\x}_i} = \Cov{\erro^{\uh}_i}{\erro^{\uh}_i} = \Cov{\erro^{\u}_i}{\erro^{\u}_i},
\end{equation}
and the projection assumption \cref{eqn:MFEnKF-forecast-step-projection} is supported. This choice leads to $\Cov{\erro^{Z}_i}{\erro^{Z}_i} = (1/2)\Cov{\erro_i}{\erro_i}$, and requires updating the observation error covariance in the Kalman gain calculation \eqref{eqn:MFEnKF-Kalman-gain} accordingly. Since the same Kalman gain is used for all variables, this change is not optimal for the analysis of principal, control, and ancillary variates.
If $s_i = 2$ then $\Cov{\erro^{Z}_i}{\erro^{Z}_i} = \Cov{\erro_i}{\erro_i}$, and the inference performed on the total and principal variates has the correct observation error. Moreover, the control and ancillary variates share the same observation error covariance. However, the projection assumption \cref{eqn:MFEnKF-forecast-step-projection} is unsupported.

A third approach would be to slightly relax the projection assumption \cref{eqn:MFEnKF-forecast-step-projection}, by allowing perturbations to it, and by additionally weakening the correlation structure of $\erro^{\x}_i$ and $\erro^{\uh}_i$. Such methods are outside the scope of this paper.

In the remainder of the paper we  primarily focus on method (i) where the Kalman gain is consistent for the total variate, and that the projection assumption is satisfied. Analysis of the ancillary variate uses a suboptimal noise level in the Kalman gain.

\begin{remark}
The control and ancillary variates have to have the same mean, and that the control variate needs to remain strongly correlated with the principal variate. In order to satisfy the first condition, we perform a re-centering procedure around the mean of the combined analysis \cref{eqn:Z-analysismean}:
\begin{equation}
\label{eqn:MFEnKF-mean-correction}
    \MeanE{\xa_i} \xleftarrow{} \MeanE{\za_i},\qquad 
    \MeanE{\uha_i} \xleftarrow{} \*\Phi_r^*\, \MeanE{\za_i},\qquad 
    \MeanE{\ua_i} \xleftarrow{} \*\Phi_r^*\,\MeanE{\za_i}.
\end{equation}%
The approach \eqref{eqn:MFEnKF-mean-correction} is not the only way to ensure that $U$ and $\uh$ have a common mean, however it is a natural choice. An alternative approach is to not correct the mean of the ancillary ensemble at all, but re-center the control ensemble: $\MeanE{\uha_i} \xleftarrow{} \MeanE{\ua_i}$.
In this second approach one runs two Kalman filters side by side, one for $X$ and one for $U$.
%
%For the second condition, the analysis control variate is replaced by the projection of the principal variate \eqref{eqn:MFEnKF-forecast-step-projection}
%%
%\begin{equation*}
%\begin{split}
%\uha_i &=\uhb_i - \*\Phi_r^*\,\widetilde{\mathbf{K}}_i\,(\Hobs_{r,i}(\uhb_i)  - \Obs_i + \,\erro^{\uh}_i)\quad  \rightarrow \\
%\uha_i &=\*\Phi_r^*\,\xb_i - \*\Phi_r^*\,\widetilde{\mathbf{K}}_i\,(\Hobs_i(\xb_i)  - \Obs_i + \erro^{\x}_i).
%\end{split}
%\end{equation*}
\end{remark}
The MLEnKF anomaly updates are as follows:
\begin{equation}
\begin{aligned}
    \An{\xa_{i}} &= \An{\xb_{i}} - \widetilde{\mathbf{K}}_i\,\left(\Z{\xb_i} - \An{\erro^{\x}_i}\right),\\
    \An{\uha_i} &= \*\Phi_r^*\,\An{\xa_{i}}, \\ %\An{\uhb_i} - \*\Phi_r^*\,\widetilde{\mathbf{K}}_i \,\left(\Z{\puhb_i} - s_i\,\An{\erro^{\x}_i}\right),\\
    \An{\ua_i}  &= \An{\ub_i} - \*\Phi_r^*\,\widetilde{\mathbf{K}}_i\,\left(\Z{\pub_i} - \An{\erro^{\u}_i}\right).
\end{aligned}
\end{equation}
Note that the anomaly updates for $\An{\ua_i}$ are done solely in reduced (and observation) space. The only significant additional cost in the analysis step is the calculation of the statistical Kalman gain \eqref{eqn:MFEnKF-Kalman-gain}.

\subsection{Statistical Analysis of the MFEnKF.}
%%%%%%%%%%%%%%%%%%%%%%%%%
For analysis we consider the case where all uncertainties are Gaussian, and the observation operators are linear, $\Hobs = \HH$, $\Hobs_r = \HH\,\*\Phi_r$. We assume that $\Mean{\xb} = X^\|t$ and $\Mean{\ub} = \Mean{\uh^\|b}$. We consider the transformed total variate and principal variate parametrized by the Kalman gain,
\begin{gather}
    \za(\mathbf{K}) = \zb - \mathbf{K}\left(\overline{\Hobs}(\zb) - Y\right),\quad \xa(\mathbf{K}) = \xb - \mathbf{K}\left(\Hobs(\xb) - Y\right),
\end{gather}
and denote by $\mathbf{K}_Z$,  $\mathbf{K}_X$ be the optimal gains given by Kalman's formula that minimize the covariances of $\za(\mathbf{K}_Z)$ and $\xa(\mathbf{K}_X)$, respectively.

\begin{theorem}
Under the assumption that the statistical Kalman gain is independent of all other variates, the MFEnKF analysis is an unbiased estimator:
\[
\Mean{\za(\mathbf{K}_Z)} = \Mean{\xa(\mathbf{K}_Z)} = \Mean{\xa(\mathbf{K}_X)} = X^\|t.
\]
\end{theorem}
\begin{proof}
The result follows from applying a Kalman formula and taking means.
\end{proof}
The following theorem shows that performing the exact analysis in the total variate leads to better estimates than performing the analysis in the principal variate.
\begin{theorem}
\label{eqn:MSKF-covariance}
The analysis total variate is $\za(\mathbf{K}_Z)$, and the principal component of the analysis total variate is $\xa(\mathbf{K}_Z)$. Application of the Kalman filter to the principal variate leads to the analysis $\xa(\mathbf{K}_X)$. It holds that:
\begin{equation}
\label{eqn:covariance-hierarchy}
    \Cov{\za(\mathbf{K}_Z)}{\za(\mathbf{K}_Z)} \leq \Cov{\xa(\mathbf{K}_X)}{\xa(\mathbf{K}_X)} \leq \Cov{\xa(\mathbf{K}_Z)}{\xa(\mathbf{K}_Z)},
\end{equation}
where inequalities are interpreted in the symmetric positive definite matrix sense.
\end{theorem}

\begin{proof}
From the optimality of the Kalman filter we have that:
\[
\Cov{\za(\mathbf{K}_Z)}{\za(\mathbf{K}_Z)} \le \Cov{\za(\mathbf{K_X})}{\za(\mathbf{K_X})}, \qquad
\Cov{\xa(\mathbf{K}_X)}{\xa(\mathbf{K}_X)} \le \Cov{\xa(\mathbf{K_Z})}{\xa(\mathbf{K_Z})}, 
\]
which proves the second inequality in \eqref{eqn:covariance-hierarchy}.
From \eqref{eqn:covariance-total-variate} we have that:
\begin{equation}
\begin{split}
    \Cov{\zb}{\zb} 
    &= \Cov{\xb}{\xb} - \Cov{\xb}{\uh^\|b}{\bigl(\Cov{\uh^\|b}{\uh^\|b} + \Cov{\ub}{\ub}\bigr)}^{-1}\Cov{\uh^\|b}{\xb}
    \le \Cov{\xb}{\xb}.
\end{split}
\end{equation}

Next, we use the above equations and the Kalman analysis covariance formula to prove the first inequality in \eqref{eqn:covariance-hierarchy}:
\begin{equation}
\label{eqn:cov-za-exact}
\begin{split}
 &   \Cov{\za(\mathbf{K}_Z)}{\za(\mathbf{K}_Z)} \le \Cov{\za(\mathbf{K}_X)}{\za(\mathbf{K}_X)} \\
    &= \left(\*I - \mathbf{K}_X\*H\right)\Cov{\zb}{\zb}\left(\*I - \mathbf{K}_X\*H\right)^\intercal + \mathbf{K}_X\,
    \Cov{\erro}{\erro}\,\mathbf{K}_X^\intercal \\
%    &= \left(\*I - \mathbf{K}_X\*H\right)\,\Cov{\xb}{\xb}\,\left(\*I - \mathbf{K}_X\*H\right)^\intercal + \mathbf{K}_X\,\Cov{\erro}{\erro}\,\mathbf{K}_X^\intercal\\
%    &\quad -  \left(\*I - \mathbf{K}_X\*H\right)\,\Cov{\xb}{\uh^\|b}{\left(\Cov{\uh^\|b}{\uh^\|b} + \Cov{\ub}{\ub}\right)}^{-1}\Cov{\uh^\|b}{\xb}\,\left(\*I - \mathbf{K}_X\*H\right)^\intercal \\
    &= \Cov{\xa(\mathbf{K}_X)}{\xa(\mathbf{K}_X)}\\
   &\quad  -  \left(\*I - \mathbf{K}_X\*H\right)\,\Cov{\xb}{\uh^\|b}{\left(\Cov{\uh^\|b}{\uh^\|b} + \Cov{\ub}{\ub}\right)}^{-1}\Cov{\uh^\|b}{\xb}\,\left(\*I - \mathbf{K}_X\*H\right)^\intercal.
\end{split}
\end{equation}
\end{proof}

We next turn our attention to sampling errors.

\begin{theorem}
\label{thm:mean-cov-mfenkf}
 Assume that EnKF produces $N_X$ i.i.d. samples of $\xa(\mathbf{K}_X)$. The covariance of the sample mean estimate about the true state is
\begin{subequations}
\begin{equation}
\label{eqn:covariance-enkf-estimate}
%\MeanE{\za} \sim \mathcal{N}\left( X^\|t, N_X^{-1}\, \Cov{\za}{\za}\right), \quad
%\MeanE{\xa} \sim \mathcal{N}\left( X^\|t, N_X^{-1}\, \Cov{\xa}{\xa}\right),
\Cov{\MeanE{\xa(\mathbf{K}_X)}}{\MeanE{\xa(\mathbf{K}_X)}} = N_X^{-1}\, \Cov{\xa(\mathbf{K}_X)}{\xa(\mathbf{K}_X)}.
\end{equation}
\end{subequations}
Assume that MFEnKF produces $N_X$ i.i.d. samples of $\xa(\mathbf{K}_Z)$ and $\uh^\|a$, and $N_U$ i.i.d. samples of $\ua$. Since
\begin{equation}
\begin{split}
    \MeanE{\za(\mathbf{K}_Z)} &= 
    (\*I - \*K_Z\*H)\left[\MeanE{\xb} - \*S\,\left(\MeanE{\uhb} - \MeanE{\ub}\right)\right]  + \*K_Z\left(Y + \erro\right)\\
    &\approx\MeanE{\xa(\mathbf{K}_Z)} - \*S\,\MeanE{\uh^\|a} + \*S\,\MeanE{\ua},
    \end{split}
\end{equation}%
and we estimate the moments of $\xa$, $\uh^\|a$ using $N_X$ samples, and the moments of $\ua$ using $N_U$ samples, then the sample mean of the analysis total variate has less variance than the Kalman filter applied to the principal variate,
\begin{equation}
    \Cov{\MeanE{\za(\mathbf{K}_Z)}}{\MeanE{\za(\mathbf{K}_Z)}} \leq \Cov{\MeanE{\xa(\mathbf{K}_X)}}{\MeanE{\xa(\mathbf{K}_X)}}.
\end{equation}
\end{theorem}
\begin{proof}
A direct calculation shows that the variance of the empirical mean estimate about the truth is:  
\begin{equation}
\label{eqn:covariance-MFEnKF-estimate}
\begin{split}
& \Cov{\MeanE{\za(\mathbf{K}_Z)}}{\MeanE{\za(\mathbf{K}_Z)}} 
    = N_X^{-1}\,\Cov{\za(\mathbf{K}_Z)}{\za(\mathbf{K}_Z)}\\
    &\quad
    + \left(N_U^{-1} - N_X^{-1}\right)(\*I - \*K_Z\*H)\,\*S\,\Cov{\ub}{\ub}\,\*S^\intercal{(\*I - \*K_Z\*H)}^\intercal \\
& \le \Cov{\MeanE{\xa(\mathbf{K}_X)}}{\MeanE{\xa(\mathbf{K}_X)}} + \left(N_U^{-1} - N_X^{-1}\right)(\*I - \*K_X\*H)\,\*S\,\Cov{\ub}{\ub}\,\*S^\intercal{(\*I - \*K_X\*H)}^\intercal \\
    &\quad  - N_X^{-1}\, \left(\*I - \mathbf{K}_X\*H\right)\,\*S\,{\left(\Cov{\uh^\|b}{\uh^\|b} + \Cov{\ub}{\ub}\right)}\,\*S^\intercal\left(\*I - \mathbf{K}_X\*H\right)^\intercal,\\
    &= \Cov{\MeanE{\xa(\mathbf{K}_X)}}{\MeanE{\xa(\mathbf{K}_X)}}\\
    &\quad - \left(\*I - \mathbf{K}_X\*H\right)\,\*S\,{\left(N_X^{-1}\Cov{\uh^\|b}{\uh^\|b} + \left(2 N_X^{-1} - N_U^{-1}\right)\Cov{\ub}{\ub}\right)}\,\*S^\intercal\left(\*I - \mathbf{K}_X\*H\right)^\intercal,
\end{split}
\end{equation}
where for the inequality we used \eqref{eqn:cov-za-exact} and \eqref{eqn:covariance-enkf-estimate}.
\end{proof}

Theorem \eqref{thm:mean-cov-mfenkf} shows that MFEnKF provides an estimate that is always at least as good as the corresponding EnKF estimate for the same number $N_X$ of high fidelity model runs. The difference comes from the smaller variance of $\za(\mathbf{K}_Z)$ compared to $\xa(\mathbf{K}_X)$ (first term in \cref{eqn:covariance-MFEnKF-estimate}), from the use of control variates in covariance estimates  and from using the data to assimilate the reduced space variables (second term in \cref{eqn:covariance-MFEnKF-estimate}).

\begin{remark}
EnKF produces an ensemble that quantifies the posterior uncertainty in the FOM state. From \eqref{eqn:covariance-hierarchy}, the posterior ensemble of principal variables $\{\xa[e](K_Z)\}_{e=1,\dots,N_X}$ constructed by MFEnKF provides (only) an upper bound for the analysis state error covariance. For posterior uncertainty quantification one can use $N_X$ members of the $\ua[e]$ ensemble to construct an ensemble of total variates.
\end{remark}

%%%%%%%%%%%%%%%%%%%%%%%%%
\subsection{Cost Analysis of the MFEnKF.}
%%%%%%%%%%%%%%%%%%%%%%%%%

We seek to find an equivalent EnKF running an ensemble size of $M_X$ full order models that gives the same analysis sampling error as MFEnKF with $N_X$ full order and $N_U$ reduced order ensemble sizes. We measure sampling errors by the trace generalized variance $\sigma_W = \tr(\Cov{W}{W})$.

By \cref{eqn:covariance-enkf-estimate} the sampling error for EnKF  is $M_X^{-1} \sigma_X$, and 
by \cref{eqn:covariance-MFEnKF-estimate} sampling error for  MFEnKF is 
$N_{\x}^{-1}\sigma_Z + \left(N_U^{-1} - N_{\x}^{-1}\right)\sigma_{\*S\, U}$. By matching these generalized variances
the effective ensemble size of the EnKF is :
\begin{equation}
    M_X = \frac{ N_X N_U \sigma_X}{N_U \sigma_Z - \sigma_{\*S U} (N_U - N_X)}.
\end{equation}
We see by direct calculation that $M_X \geq N_X$ whenever $N_U \geq N_X$ and $\sigma_Z \geq \sigma_{\*S\, U} (N_U - N_X)/N_U$.
%Observe additionally that by simple manipulation, and, assuming that $N_U \geq N_X$ and $\sigma_Z \geq \sigma_{\*S U} \frac{N_U - N_X}{N_U}$,
%\begin{equation}
%    M_X = \frac{ N_X N_U \sigma_X}{N_U \sigma_Z - \sigma_{\*S U} (N_U - N_X)} \geq  \frac{ N_X N_U \sigma_X}{N_U \sigma_X - \sigma_{\*S U} (N_U - N_X)} \geq N_X.
%\end{equation}

Let $C_X$ be the cost of running a full order model, and $C_U$ the cost of running a lower fidelity model is $C_U$. To obtain similar analyses, the cost of running the EnKF is $C_X\, M_X$, and the cost of running the MFEnKF is $C_X N_X + C_U(N_X + N_U)$. Consequently,  the MFEnKF algorithm is more efficient than EnKF whenever the cost of running the lower fidelity model satisfies:
%
%We can derive the cost of the lower fidelity model, $C_U$ for which the MFEnKF algorithm will be beneficial to run over the EnKF.
%We require,
%\begin{equation}
%    C_X N_X + C_U(N_X + N_U) \leq C_X M_X,
%\end{equation}
%which has the closed form solution,
\begin{equation}
    C_U \leq \frac{C_X (M_X - N_X)}{N_X + N_U}.
\end{equation}

%%%%%%%%%%%%%%%%%%%%%%%%%
\subsection{Telescopic extension.}
%%%%%%%%%%%%%%%%%%%%%%%%%

We now discuss the telescopic extension from the two-fidelity to the multifidelity ensemble Kalman filter, by utilizing the multivariate control variate extensions discussed in \cref{sec:multifidelity}. Consider a sequence of projection operators, $\*\Phi_{r_\ell}^{\ell}$ and $\*\Phi_{r_\ell}^{\ell,*}$ for $\ell = 1,\dots,\!L$, and denote $\overline{\*\Phi}_{r_\ell}^{\ell} = \prod_{\lambda=1}^{\ell} \*\Phi_{r_\lambda, \lambda}$. The control variate relation between fidelity $\ell-1$ and $\ell$ is $\uh_{\ell,i}=\*\Phi_{r_\ell}^{\ell,*}\,\u_{\ell-1,i}$, with $\u_{0,i}\ \equiv \x_i$. The corresponding gain from fidelity $\ell$ to the fidelity of the principal variate is $\overline{\*S}_\ell = 2^{-\ell}\, \overline{\*\Phi}_{r_\ell}^{\ell}$. The random variables representing the control variate and the ancillary variate in the two-fidelity scheme now represent the corresponding first fidelity variates. Extending the total variate to $\!L$ fidelities gives:
\begin{equation}
    \zb_i = \xb_i - \sum_{\ell=2}^{\!L} 2^{-\ell}\,\overline{\*\Phi}_{{r_\ell},\ell}\,(\uhb_{\ell,i} - \ub_{\ell,i}).
\end{equation}
The empirical Kalman gain is computed through a natural extension o the two-fidelity approach.
The MFEnKF anomaly updates are defined as:
\begin{equation}
\begin{aligned}
%    \An{\uha_{\ell,i}} &= \An{\uhb_{\ell,i}} - \overline{\*\Phi}_{r_\ell}^{\ell,*} \widetilde{\mathbf{K}}_i \left(\Z{\overline{\*\Phi}_{r_\ell}^{\ell}\uhb_{\ell,i}} - s_i^{\ell-2}\An{\erro^{\u_{\ell-1}}_i}\right),\\
    \An{\ua_{\ell,i}} &= \An{\ub_{\ell,i}} - \overline{\*\Phi}_{r_\ell}^{\ell,*}\widetilde{\mathbf{K}}_i \left(\Z{\overline{\*\Phi}_{r_\ell}^{\ell}\ub_{\ell,i}} - \An{\erro^{\u_\ell}_i}\right).
\end{aligned}
\end{equation}
The additive perturbed observation errors are chosen in a fashion similar to the methods described above. Note that for a large number of fidelities, from a practical perspective, it might be beneficial to choose method (ii) with $s_i = 1$, thereby making the synthetic observation error equal for all variates, at the cost of the total variate observation error being reduced to $\Cov{\erro^Z_i}{\erro^Z_i} = ((1+2^{1-2\!L})/3)\,\Cov{\erro_i}{\erro_i}$.

%The additive observation errors are chosen such that $\erro^{\uh_\ell}_i = s_i\,\erro^{\u_{\ell-1}}_i$, $\Cov{\erro^{U_{\ell}}_i}{\erro^{U_{\ell}}_i} = s_i^2\,\Cov{\erro^{U_{\ell-1}}_i}{\erro^{U_{\ell-1}}_i}$, and consequently $\Cov{\erro^Z_i}{\erro^Z_i} = (1-s_i/2+s_i\,(s_i^2/4)^\!L)/(1+s_i/2)\,\Cov{\erro_i}{\erro_i}$. For $s_i=1$ we have $\Cov{\erro^Z_i}{\erro^Z_i} = ((1+2^{1-2\!L})/3)\,\Cov{\erro_i}{\erro_i}$.
%%%%%%%%%%%%%%%%%
%%%%%%%%%%%%%%%%%%%%%%
\section{The test model hierarchy.}
\label{sec:models}
%%%%%%%%%%%%%%%%%%%%%%
One salient feature of our MFEnKF framework is that it can employ different spaces to represent the models at different resolutions.
In our numerical tests we employ the following models of the quasi-geostrophic equations (QGE). The highest resolution model, called the \textit{truth}, represents the reference solution and provides $\xt_i$ and the observation data via \cref{eqn:uobs}. In~\cref{sec:dns} the truth corresponds to a direct numerical simulation (DNS) on a fine mesh. The FOM is an accurate approximation of the truth, and is obtained in \cref{sec:high-res} by performing DNS on a coarser spatial mesh. The ROM is a low cost approximation of the FOM, and is obtained in~\cref{sec:pod} by performing a POD in the FOM space, then reducing the number of modes that represent the dynamics.  

\begin{figure}[t]
\centering
  \includegraphics[width=1\linewidth]{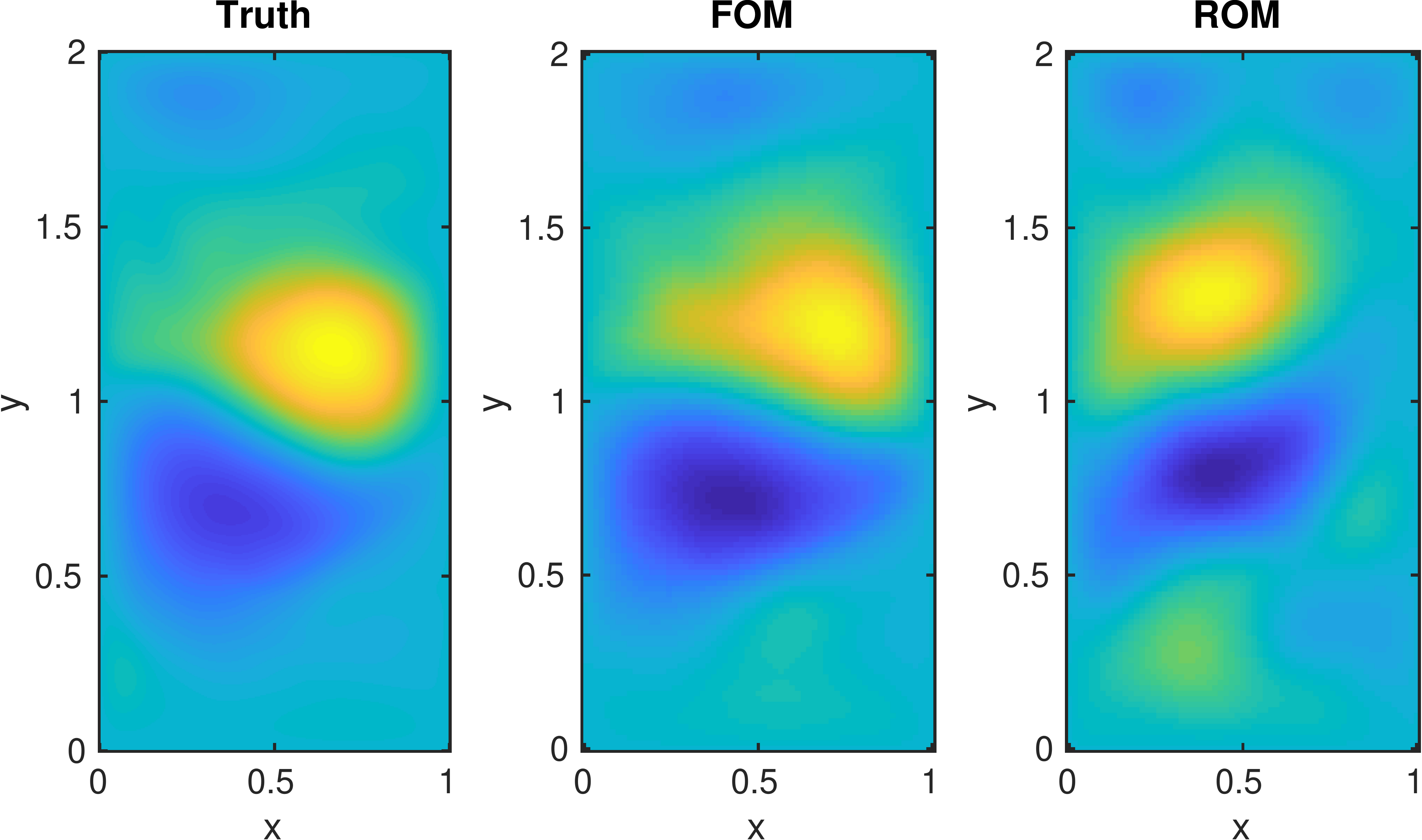}
  \caption{A ten days forecast of of the quasi-geostrophic equations. A perturbation of the true DNS through both the high-res FOM, and the low-res ROM ($r=50$) compared to the unperturbed reference truth. The plotted values represent the streamfunction, with yellow representing positive values and blue representing negative values.}
  \label{fig:qgandrom}
\end{figure}

Figure \cref{fig:qgandrom} presents a  comparison between the truth, the FOM, and the ROM solutions for a ten days forecast with QGE. All discrete models are implemented in our test suite~\cite{otpsoft,otp}.
The DNS computational cost is the highest, the high-res FOM computational cost is $130$ times lower than the DNS cost, and the low-res ROM ($r=50$) computational cost is $63$ times lower than the FOM cost.

%%%%%%%%%%%%%%%%%%%%%%
\subsection{The quasi-geostrophic equations (QGE).}
\label{sec:qge}
%%%%%%%%%%%%%%%%%%%%%%
%
%QGE~\cite{foster2013finite,foster2013two,foster2016conforming,kim2015bspline,ferguson2008numerical}, 
The QGE~\cite{foster2013finite,ferguson2008numerical,MW06,greatbatch2000four}
are based on the barotropic vorticity equations, and are widely used in both data assimilation and reduced order modeling, thereby providing an excellent test problem for MFEnKF.  Here we follow the formulation given in~\cite{mou2020data,san2015stabilized}:
\begin{equation}
  \begin{gathered}
    \omega_t + J(\psi,\omega) - \mathrm{Ro}^{-1}\psi_x = \mathrm{Re}^{-1}\Delta\omega + \mathrm{Ro}^{-1}F, \\[0.25em]
    J(\psi,\omega) \equiv \psi_y \omega_x - \psi_x \omega_y,\quad \omega = -\Delta\psi,
  \end{gathered}
  \label{eqn:qge}
\end{equation}
where $\omega$ is the vorticity, $\psi$ is the streamfunction, $\mathrm{Re}$ is the Reynolds number, $\mathrm{Ro}$ is the Rossby number, and $F$ is a forcing term. 
%The relation between streamfunction and vorticity is enforced by the last equation.  
We use a symmetric double gyre for the forcing term~\cite{greatbatch2000four,mou2020data,san2015stabilized}
%
%\begin{equation}
$
  F = \sin\left(\pi(y-1)\right),
$
%\end{equation}
%
and homogeneous Dirichlet 
boundary conditions 
%
%\begin{equation}
$
    \omega(x, y) = 0, \,
    \psi(x, y) = 0,\, (x,y)\in\partial\Omega,
$
%\end{equation}
%
where  the computational domain is $\Omega = [0, 1] \times [0, 2]$. The constants are $\mathrm{Re}=450$ and $\mathrm{Ro}=0.0036$. The time scale of the problem uses $80$ time units to represent $20.12$ years~\cite{san2015stabilized}.

\subsection{The  direct numerical simulation.}
    \label{sec:dns}

The truth involves a DNS simulation of the QGE~\eqref{eqn:qge} on a  `fine' spatial mesh with $255$ interior points in the $x$ direction and $511$ points in the $y$ direction. Second order finite difference discretization are used for both first and second order spatial derivatives, together with the Arakawa approximation~\cite{arakawa1966computational,jespersen1974arakawa} for the Jacobian term $J$ in \cref{eqn:qge}. The embedded Poisson equation is solved using a precomputed sparse Cholesky decomposition. 

Time integration for this, and subsequent discretizations is performed using a fourth order `Almost Runge-Kutta' method with adaptive time stepping~\cite{rattenbury2005almost}.  We take observations every 24 hours (approximately $0.0109$ model time units) of $150$ equally-spaced variables. 

%%%%%%%%%%%%%%%%%%%%%%
\subsection{The full order model.}
\label{sec:high-res}
%%%%%%%%%%%%%%%%%%%%%%

The FOM performs a numerical simulation of the QGE~\eqref{eqn:qge} on a `coarse' spatial mesh with $63$ interior points in the $x$ direction and $127$ points in the $y$ direction.
The same spatial and temporal discretizations as for the truth simulation are used.
As illustrated in Fig.~\ref{fig:qgandrom}, although the FOM approximation is qualitatively similar to the DNS approximation, the former does not capture all the physical details displayed by the latter.
%\blue{check!}
The changes of grid (from the DNS to FOM state-space) are performed through canonical multigrid techniques~\cite{zubair2009efficient}.

%%%%%%%%%%%%%%%%%%%%%%
% \subsection{The POD-Galerkin reduced order numerical model (ROM-POD).}
\subsection{The reduced order model (ROM).}
\label{sec:pod}
%%%%%%%%%%%%%%%%%%%%%%

%For our low-res model we will take a reduced order model construction of the QGE.
The construction of ROM for the QGE~\eqref{eqn:qge} follows~\cite{mou2020data,san2015stabilized,strazzullo2018model}.
We start by building the ROM vorticity basis using the proper orthogonal decomposition (POD)~\cite{HLB96} and 
%We review the POD-Galerkin algorithm through 
the method of snapshots~\cite{sirovich1987turbulence1}. We collect $\omega_1, \dots, \omega_M$ snapshots of FOM vorticity at $M=700$  different times along a model  trajectory, with each snapshot  roughly 6 months apart in model time. 
%\blue{What does the next sentence mean? Please explain.}
The snapshot trajectory is unrelated to the trajectory of the truth  
in order to simulate more realistic operational conditions. 
We build the snapshot covariance matrix
%
%\begin{equation}
$ \left[\!C\right]_{ij} = \inner{\omega_i}{\omega_j}$,
  $i, j = 1, \ldots, M$,
%\end{equation}
%
using a quadrature approximate integration. The eigendecomposition of $\!C$ yields the ROM vorticity basis $\{ \varphi_1, \ldots, \varphi_r \}$, where $r$ is the ROM dimension. 
The relative kinetic energy of the first $r$ modes~\cite{mou2020data} is listed in \cref{tab:kinetic-energy}, where the relative kinetic energy is calculated based on FOM data over the time $[10,80]$ (units). In numerical simulations we consider 
$r=10$, $25$, and $50$.

%POD basis:
%
% \begin{equation}
% \!C = \!V\*\Lambda\!V^\intercal, \qquad
%     \varphi_j = \frac{1}{\lambda_j}\sum_{i=1}^M v_i^j \omega_i.
% \end{equation}
%
%where $v_i^j$ is the $j$th component of the $i$th eigenvector of $\!C$, and $\varphi_j$ is the $j$th POD basis function.
The ROM streamfunction basis is  obtained from the ROM vorticity basis 
%space will be defined 
by the relationship 
%\begin{equation}
$
    -\Delta \phi_i = \varphi_i, 
    \ i = 1, \ldots, r.
$
%\end{equation}
%
The ROM vorticity and streamfunction approximations are
\begin{equation}
    \widetilde{\omega}(t) 
    = \sum_{i=1}^r a_i(t)\varphi_i
    \qquad  \text{and} \qquad 
    \widetilde{\psi}(t) 
    = \sum_{i=1}^r a_i(t)\phi_i,
    \label{eqn:rom-expansions}
\end{equation}
respectively.
The dynamics of the unknown ROM coefficients $\*a(t) = [a_1(t),\dots,a_r(t)]^\intercal$ is determined by using a Galerkin projection of the equations ~\eqref{eqn:rom-expansions}: 
%i.e., by utilizing the ROM expansions~\eqref{eqn:rom-expansions} in the first equation in~\eqref{eqn:qge} and then projecting the resulting equation onto the ROM vorticity space. 
% We are now ready to define the quadratic Galerkin POD-ROM.
% Assume that we take the first $r<M$ POD basis vectors. If $\psi(t)$ is our streamfunction in full space at time $t$, then
% \blue{How did we get the second equality below?}
% %
% \begin{equation}
%   \widetilde{\psi}(t) = \sum_{i=1}^r \inner{\psi(t)}{\phi_i}\phi_i = \sum_{i=1}^r \underbrace{\inner{-\Delta\psi(t)}{\varphi_i}}_{ \eqqcolon [a(t)]_i}\,\phi_i,
% \end{equation}
% %
% is its POD-space approximation, which can be represented by its coefficients, $[a(t)]_i$.
%This yields the {\it ROM} for the ROM-space state $\*a(t)$:
%The components of the POD-Galerkin model are propagated forward in time with a quadratic model,
%
\begin{equation}
  \*a_t = \*b + \*A\*a + \*a^\intercal\,\!B\,\*a,
  \label{eqn:pod-rom}
\end{equation}
where
%
%which represents the {\it POD-Galerkin ROM}.
%
%In~\eqref{eqn:pod-rom}, 
%$\*b$ is a vector corresponding to the forcing term, $\*A$ is a matrix representing the linear action, and $\!B$ is a 3-tensor representing the Jacobian term:
%
\begin{equation}
\begin{aligned}
&  {[\*b]}_i = \mathrm{Ro}^{-1}\inner{F}{\varphi_i},\qquad 
  {[\*A]}_{i,j} = \mathrm{Ro}^{-1}\inner{\frac{\partial \phi_j}{\partial x}}{\varphi_i}-\mathrm{Re}^{-1}\inner{\nabla\varphi_j}{\nabla\varphi_i},\\
&  {[\!B]}_{i,m,n} = -\inner{J(\varphi_m,\phi_n)}{\varphi_i},
\end{aligned}
\end{equation}
with the inner products implemented using the 2D Simpson's rule discretization.
As illustrated in Fig.~\ref{fig:qgandrom}, although the ROM approximation is qualitatively similar to the DNS and FOM, but does not capture all the physical details.

\begin{table}[ht]
\centering
\begin{tabular}{|c | c c c c|} 
 \hline
 $r$ & $10$ & $25$ & $50$ & $100$ \\ [0.5ex] 
 \hline
 %Relative KE & $0.8630$ &   $0.9336$  &  $0.9604$ & $0.9726$  \\ [0.5ex]
 Relative KE & $0.9071$ &   $0.9679$  &  $0.9871$ & $0.9963$  \\ [0.5ex]
 \hline
\end{tabular}
\caption{Relative kinetic energy for the first $r$ ROM modes.}
\label{tab:kinetic-energy}
\end{table}

\subsection{Projection operators.}

We now explicitly define the space projection operators from \cref{sec:podcontrol} for the QGE and its corresponding ROM.
Let the matrix $\*D$ represents the 2D Simpon's rule discretization of the spatial inner product, and $\*\Delta$ the discrete version of the Laplacian. Then $\*M_\upsilon = \*\Delta\*D\*\Delta$ in \cref{eqn:projection-operators}.
Let $\*\Psi_r$ be the $r$ dominant eigenvectors of the temporal covariance matrix $\!C$. The following operators preserve the relationship between the vorticity and streamfunction bases:
\begin{equation}
    \*\Phi_r = -\*D^{-1/2}\, \*\Delta^{-1}\, \*\Psi_r,\quad \*\Phi_r^* = -\*\Psi_r^\intercal\,\*D^{1/2} \*\Delta.
\end{equation}
\section{Numerical Experiments.}
\label{sec:numerical}
%%%%%%%%%%%%%%%%%%%%%%%%%%%%%%%%%%%%

The numerical experiments aim to investigate the performance of MFEnKF compared against other EnKF methods, to asses how the analysis  accuracy depends on the accuracy of the underlying ROM, and what the usefulness of the ensembles underlying MFEnKF to represent probability distributions of interest. 
We consider two fidelities, see \cref{sec:models}. The principal variate represents the uncertainty in QGE FOM, and our control and ancillary variates represent the uncertainty in QGE ROM. The truth is provided by the QGE DNS model.

In order to create synthetic observations the truth solution is relaxed onto the FOM space, and the states corresponding to 150 equally spaced indices are observed. The observation error covariance is $\Cov{\eta}{\eta} = \*I_{150}$.
In terms of EnKF correction techniques the experiments use inflation, as there is strong evidence~\cite{popov2020explicit} that it is an explicit probabilistic requirement in EnKF-based methods. The same inflation factor is used for the principal and control variate ensembles, and an independently chosen value is used for the ancillary ensemble.
All experiments run for $350$ observation steps. The first $50$ are discarded to account for model spinup, and the rest are used to compute the analysis quality results.

%%%%%%%%%%%%%%%%%%%%%%%%%%%%%%%%%%%%
\subsection{Comparison with other techniques.}
%%%%%%%%%%%%%%%%%%%%%%%%%%%%%%%%%%%%

\begin{figure}[t]
\centering
  \includegraphics[width=1\linewidth]{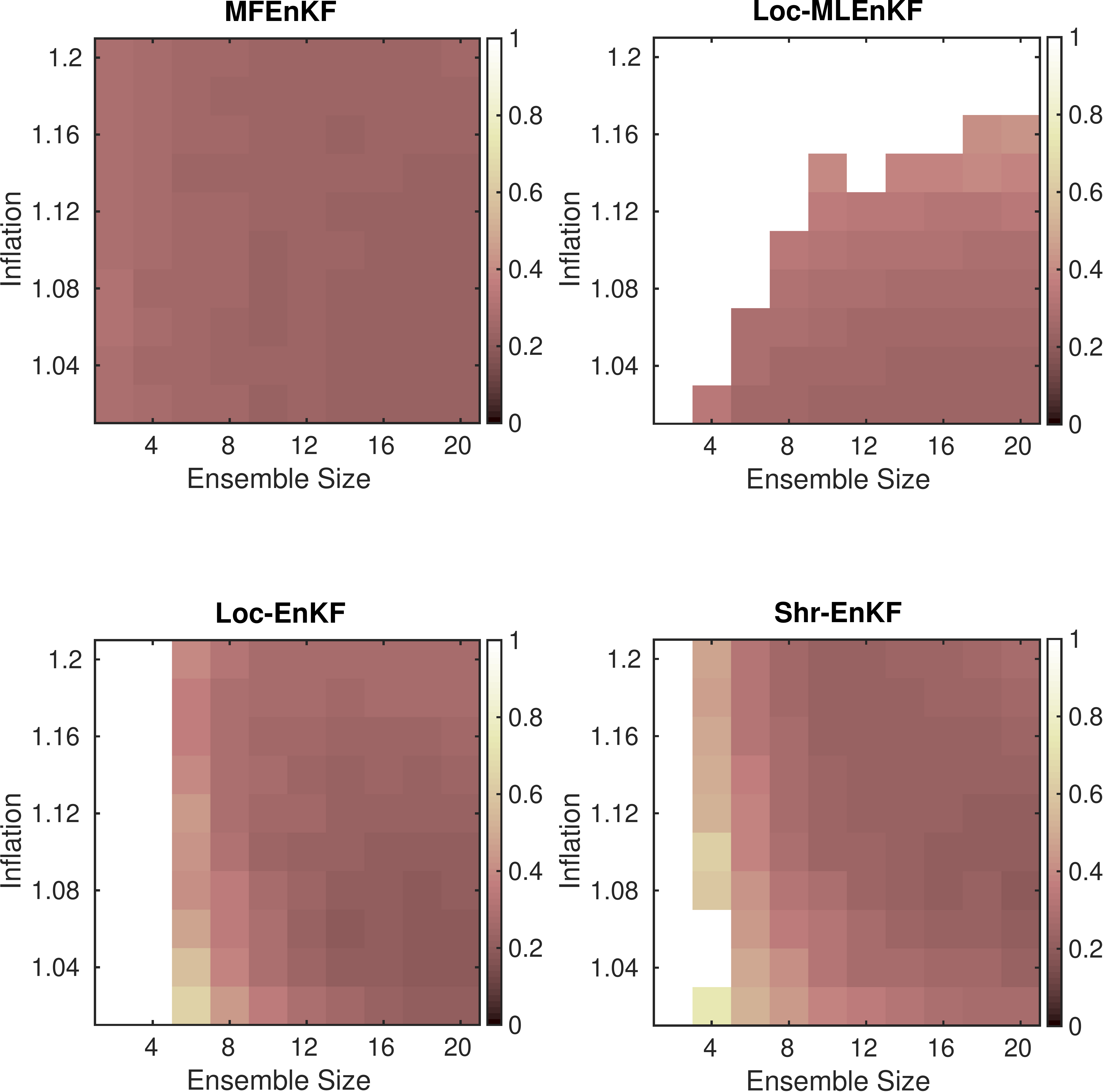}
  \caption{Analyses RMSE comparison of two-fidelity MFEnKF versus a corrected and localized MLEnKF \cite{Hoel_2016_MLEnKF,Chernov_2017_MLEnKF} (Loc-MLenKF), a localized EnKF (Loc-EnKF), and a covariance shrinkage-based EnKF (Shr-EnKF). Darker shades represent lower error, with lighter shades representing higher error.
  }\label{fig:res1}
\end{figure}

We assess the performance of MFEnKF compared to three other perturbed observation filters: a heavily corrected and localized version of the original MLEnKF \cite{Hoel_2016_MLEnKF,Chernov_2017_MLEnKF}, a localized EnKF, and a shrinkage covariance corrected EnKF.

MFEnKF uses a ROM of dimension $r=50$, and ancillary ensemble size $N_{\u}=40$, and ancillary inflation factor 
$\alpha_{\u}=1.1$.
Since the standard MLEnKF \cite{Hoel_2016_MLEnKF,Chernov_2017_MLEnKF} did not converge for this test problem, we consider a modified version  correct it by augmenting the MLEnKF formulas with the forecast correction \cref{eqn:MFEnKF-forecast-step-projection}, the mean correction \cref{eqn:MFEnKF-mean-correction}, and Gaussian kernel localization with a radius of 20 grid units. The ROM and ROM space ensembles have an identical configuration to that used with MFEnKF. All other implementation details follow \cite{Hoel_2016_MLEnKF}. 
The localized EnKF uses Gaussian localization with a radius of 20 grid units. The covariance shrinkage EnKF~\cite{Sandu_2015_covarianceShrinkage} uses the target matrix to be a snapshot-derived localized background covariance matrix and the (normalized) Rao-Blackwellized Ledoit and Wolf estimator.

For all methods we employ different FOM ensemble sizes $N_{\x}=2,4,\dots,20$ and inflation factors $\alpha_{\x}=1.02, 1.04, \dots 1.2$, and calculate the spatio-temporal RMSE (averaged over 3 model runs) of the analysis with respect to the truth. Results reported in \cref{fig:res1} show that MFEnKF outperforms the heavily corrected localized MLEnKF, meaning that our derivation of the MFEnKF from a robust control variate framework indeed has merit. We additionally outperform standard correction techniques such as localization and covariance shrinkage in standard EnKF.  The combination of a FOM and a ROM in ensemble-based methods could be used as a replacement to (or in conjunction with) such methods.

We perform a simple computational cost analysis. For $r=50$ the normalized the cost of  one ROM run is 1 unit, and the cost of one FOM run is approximately $63$ units (empirically measured wall-clock time). The cost of one MFEnKF forecast is $63 \Nens_{\x} + (\Nens_{\x} + \Nens_{\u})$ normalized wall clock units by~\cref{eqn:cost}. If $\Nens_{\x} + \Nens_{\u} \approx 63$, then the cost  roughly equals that of one extra FOM ensemble member.
For $\Nens_{\u} = 40$ and $\Nens_{\x} = 4$ we obtain a stable algorithm for the cost of about 5 FOM runs, while maintaining the accuracy of a (non-localized, not pictured) EnKF with an ensemble size $N_X = 40$, and that of a localized EnKF with an ensemble size of $N_X = 12$. This results in two-fold to eight-fold cost savings.

%%%%%%%%%%%%%%%%%
\subsection{Impact of ROM dimension.}
%%%%%%%%%%%%%%%%%

\begin{figure}[t]
\centering
  \includegraphics[width=0.5\linewidth]{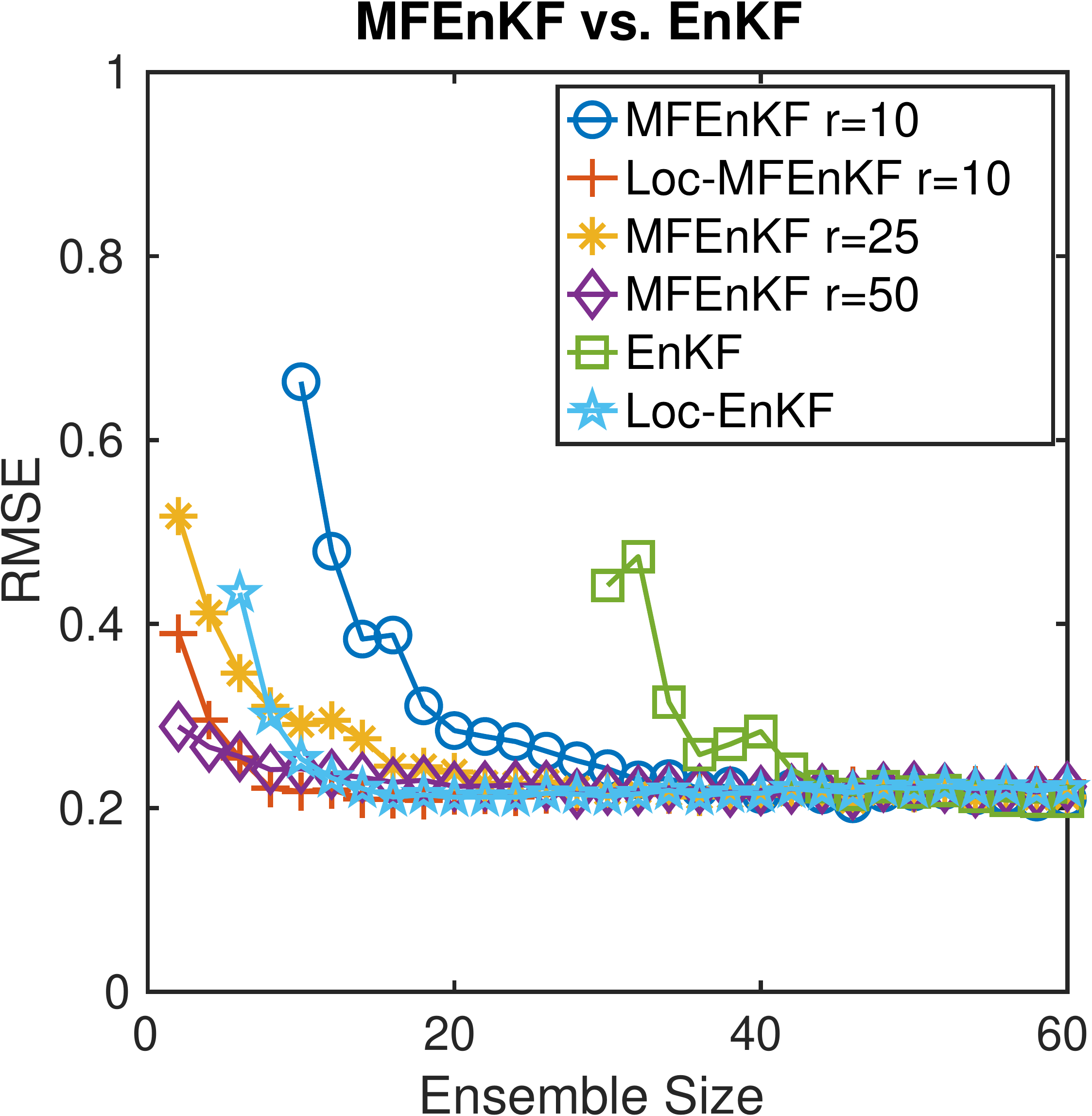}
  \caption{Comparison RMSE of the two-fidelity %POD 
  MFEnKF for various values of the ROM dimension 
  %POD mode variable 
  $r$, and both a localized and standard EnKF.}\label{fig:res3}
\end{figure}

The second numerical experiment assesses the impact of ROM basis size. We consider $r=10$, $r=25$, $r=50$, representing a severely underrepresented system, an underrepresented system, and a system with a medium level of representation, respectively. For the severely underrepresented system we use a localized (Gaussian with radius of 20 grid units) implementation.
The ROM ensemble sizes are $N_{\u}=9$, $N_{\u}=20$, and $N_{\u}=40$, respectively, in order to always have undersampled ensembles. For comparison we consider both a localized and a standard EnKF. The inflation factor $\alpha_{\x}=1.1$ is used in all experiments. Spatio-temporal analysis RMSEs (averaged over three runs) for different FOM ensemble sizes $N_{\x}$ are shown in \cref{fig:res3}. Larger ROM bases lead to more accurate analyses. Even with the particularly small basis size $r=10$ MFEnKF is significantly superior to a standard EnKF; this substantiates \cref{rem:positivity-of-covariance} that the magnitude of the analysis covariance can only be improved when an optimal gain is used, even if the quality of the ROM is poor.  A basis of size $r=25$ leads to results very similar to that of the localized EnKF, and that even using only $r=10$ basis vectors with a localized variant of the MFEnKF algorithm is almost as good as using $r=50$ basis vectors.

%%%%%%%%%%%%%%%%%
\subsection{Rank histograms.}
%%%%%%%%%%%%%%%%%

\begin{figure}[t]
\centering
  \includegraphics[width=1\linewidth]{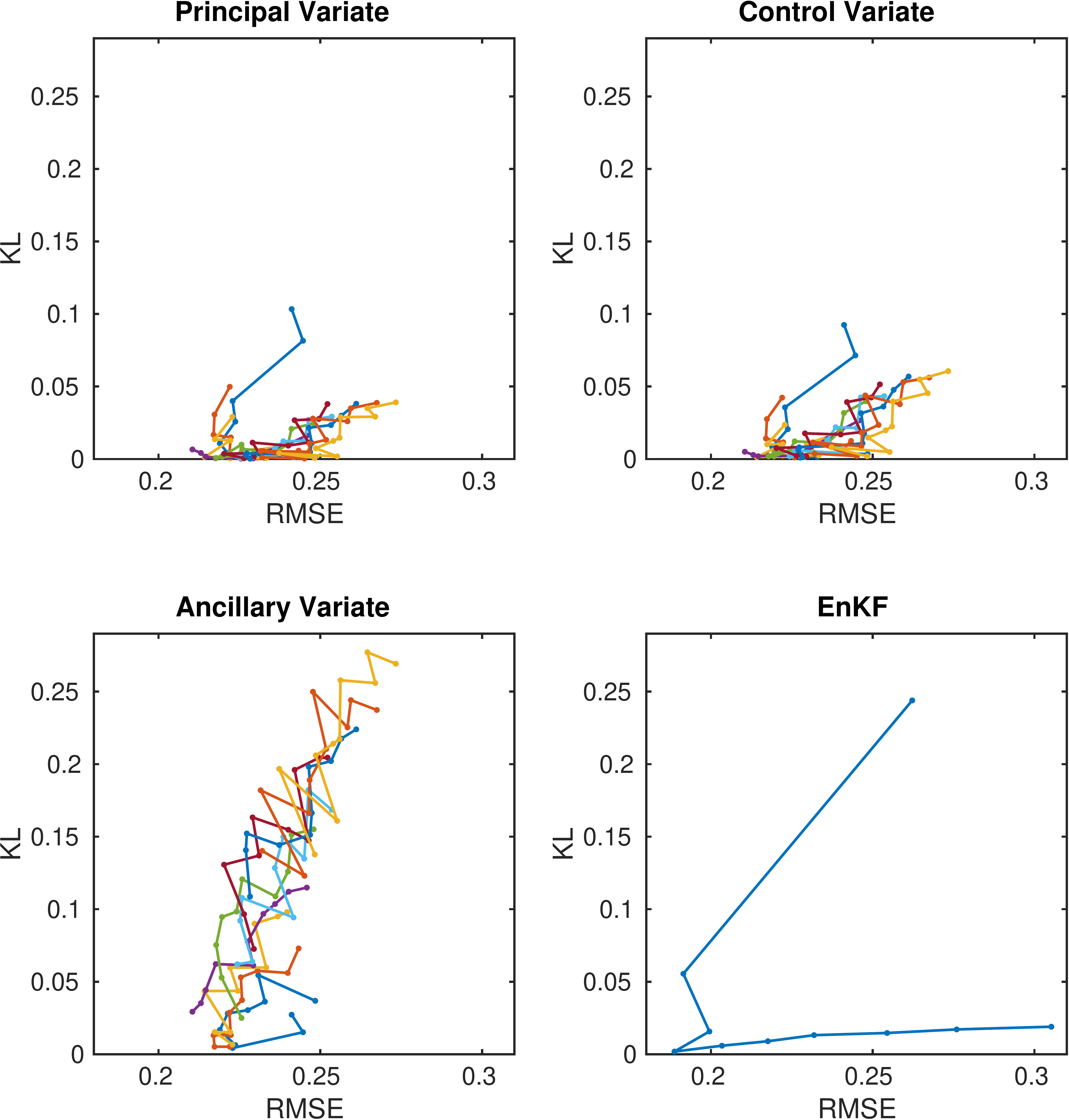}
  \caption{KL divergence (in nats) of rank histogram from uniform distribution compared with spatio-temporal RMSE. For MFEnKF results each line represents a constant value of inflation of the principal variate ensemble, and each point a different value of ancillary variate ensemble inflation. For EnKF result each point represents a different value of inflation.}
  \label{fig:reskl}
\end{figure}

A rank histogram measures the reliability with which an ensemble forecast captures the probability distributions of certain quantities of interest~\cite{hamill2001interpretation}. Consider an ensemble of scalar quantities representing independent draws the exact distribution (here, normal); tallying the number of ensemble members that underestimate each of them should result in a uniformly distributed histogram. 

We consider the rank histograms of the ensembles representing the principal variate, control variate, and the ancillary variate, and measure the KL divergence~\cite{kullback1951information} between these histograms ($Q$) and an ideal uniform distribution ($P$):
\begin{equation}
\label{eqn:KL}
    D_{KL}\left(P\middle|\middle|Q\right) = -\sum_i P_i\log\left(\frac{P_i}{Q_i}\right),
\end{equation}
where the result represents the information (in nats) required to transform one distribution to the other. A value close to zero nats implies that the distributions are essentially indistinguishable.

We construct the rank histograms using the truth values of all 150 observed variables, and assuming their independence from each other. Multiple data assimilation experiments are carried out using inflation factors from $1.02$ to $1.2$ for each of the ensembles considered herein. We compare two-fidelity MFEnKF with $r=50$, $N_{\x} = 20$, and $N_{\u} = 40$ to a vanilla perturbed observations EnKF with $N_{\x} = 60$. For each experiment, ensemble, and algorithm we collect the KL divergence value \eqref{eqn:KL} and the analysis RMSE. 

\Cref{fig:reskl} shows the KL divergence values versus RMSEs, where each point corresponds to a different experiment. It can be seen that the EnKF preserves predictability (low KL divergence value) for almost all values of inflation, and that inflation mainly affects the RMSE. In contrast, for the MFEnKF, inflation does not have such a dramatic impact, especially for the principal and control variates. For the ancillary variate, inflation plays a key role in lowering the KL divergence of the rank histogram from the normal, and has much less impact on RMSE. This means that in terms of predictability, the ensemble of the principal variates is  more reliable than that produced by EnKF.

%%%%%%%%%%%%%%%%%%%%%%%%
%!TEX root = main.tex
%%%%%%%%%%%%%%%%%
\section{Conclusions and future work.}
\label{sec:conclusions}
%%%%%%%%%%%%%%%%%

This work develops the new multifidelity ensemble Kalman filter algorithm based on a linear control variate framework.  The multivariate linear control variate theory perspective allows for rigorous multifidelity extensions of the EnKF, where the uncertainty in  coarser levels in the hierarchy of models represent control variates for the uncertainty in finer levels. Thus, complementing a small ensemble of high fidelity model runs with larger ensembles of cheaper, lower fidelity runs, results in improved analyses with only small additional computational costs.
Different models in the hierarchy can have  different state spaces, with different dimensions and/or different inner products. The mapping between different spaces (i.e., the mapping of each control variate to the space of the corresponding principal variates)  is done by gain matrices that can be computed in an optimal way.  The analysis of the new algorithm shows that it always produces better analyses than EnKF with the same number of high fidelity ensemble members.

MFEnKF has several advantages over other approaches to couple information from different models in data assimilation. Using  multimodel ensembles to build empirical covariances in EnKF faces the challenge that different ensemble members have different dimensions. The strategy of stacking models to formally construct a supermodel, and applying EnKF in the aggregated space, cannot employ different numbers of ensemble members of different models. 
MLEnKF \cite{Hoel_2016_MLEnKF,Chernov_2017_MLEnKF} applies EnKF in the high fidelity space, and uses different model levels to improve the empirical covariance estimates. Incorporating different model levels using signed empirical measures leads to  possibly non-positive-definite multilevel covariance estimates, and requires all models to share the same state space.

Numerical experiments with a quasi-geostrophic model reveal that MFEnKF provides significant improved analysis over the standard MLEnKF, and is competitive with other EnKF correction methods such as localization and covariance shrinkage. Moreover, the ensembles underlying the MFEnKF technique are useful in representing the probability distributions of given quantities of interest. 

%
%optimized for reduced order model control variates. 
%\blue{This sounds too modest.  I'd say ``We proposed a new MLEnKF algorithm..."}
%We tested this MFEnKF on the  quasi-geostrophic equations with a corresponding ROM. We have proved that, under natural assumptions, ROM operators are, naturally, optimal control variate gain operators. We have shown that even when the ROM ensembles are run outside of the snapshot window, the forecasts produced are useful as control variate estimators. We have shown that utilizing this approach it is possible to significantly reduce the number of full order model runs without sacrificing accuracy whatsoever, with only a marginal increase in the cost of the analysis step.
%\blue{
%%Overall, the MLEnKF computational cost is significantly/dramatically lower than the EnKF computational cost.
%Should we also say something about how MFEnKF is better than MLEnKF and EnKF?
%}
%We have shown that there is significant promise in using multiple kinds of ROMs as control variates in operational data assimilation.

The algorithm discussed herein is a multifidelity variant of the perturbed observations EnKF. An interesting future research direction is to develop multifidelity square root filters, e.g., multifidelity LETKF~\cite{hunt2007efficient}. Another interesting direction is extending the control variate approach to the case where different models are not hierarchically organized.
%Yet another future direction would look at ROMs that are more approximate, such as those involving tensor decompositions, or, ROMs with nonlinear space reductions, like those that can be obtained from machine learning techniques.
Proving more rigorous error bounds~\cite{dihlmann2016reduced,zerfas2019continuous,pagani2017efficient} for the new MFEnKF framework could provide further insight into parameter model choices.

%%%%%%%%%%%%%%%%%
\section*{Acknowledgments.}
%%%%%%%%%%%%%%%%%
The first author would like to thank Steven Roberts for his help and knowledge with regards to time integration methods.

%%%%%%%%%%%%%%%%%%%%%%%%

\bibliographystyle{siamplain}
\bibliography{Bib/biblio,Bib/traian,Bib/sandu,Bib/data_assim_multilevel,Bib/data_assim_kalman}

\begin{thebibliography}{10}

\bibitem{Anderson_2001_EAKF}
{\sc J.~L. Anderson}, {\em An {e}nsemble {a}djustment {K}alman {f}ilter for
  data assimilation}, Monthly Weather Review, 129 (2001), pp.~2884----2903.

\bibitem{anderson2007adaptive}
{\sc J.~L. Anderson}, {\em An adaptive covariance inflation error correction
  algorithm for ensemble filters}, Tellus A: Dynamic Meteorology and
  Oceanography, 59 (2007), pp.~210--224.

\bibitem{Anderson_2012_localization}
{\sc J.~L. Anderson}, {\em Localization and sampling error correction in {EnKF}
  data assimilation}, Monthly Weather Review, 140 (2012), pp.~2359--2371.

\bibitem{arakawa1966computational}
{\sc A.~Arakawa}, {\em {Computational design for long-term numerical
  integration of the equations of fluid motion: Two-dimensional incompressible
  flow. Part I}}, Journal of Computational Physics, 1 (1966), pp.~119--143.

\bibitem{asch2016data}
{\sc M.~Asch, M.~Bocquet, and M.~Nodet}, {\em {Data assimilation: methods,
  algorithms, and applications}}, SIAM, 2016.

\bibitem{Sandu_2016_reduced-sampling4DVar}
{\sc A.~Attia, R.~Stefanescu, and A.~Sandu}, {\em The reduced-order hybrid
  {Monte-Carlo} sampling smoother}, International Journal of Numerical Methods
  in Fluids, 83 (2016), pp.~28--51, \url{https://doi.org/10.1002/fld.4255},
  \url{http://dx.doi.org/10.1002/fld.4255}.

\bibitem{brunton2019data}
{\sc S.~L. Brunton and J.~N. Kutz}, {\em Data-driven science and engineering:
  Machine learning, dynamical systems, and control}, Cambridge University
  Press, 2019.

\bibitem{Burgers_1998_EnKF}
{\sc G.~Burgers, P.~J. van Leeuwen, and G.~Evensen}, {\em Analysis scheme in
  the {E}nsemble {K}alman {F}ilter}, Monthly Weather Review, 126 (1998),
  pp.~1719--1724.

\bibitem{cao2007reduced}
{\sc Y.~Cao, J.~Zhu, I.~M. Navon, and Z.~Luo}, {\em A reduced-order approach to
  four-dimensional variational data assimilation using proper orthogonal
  decomposition}, Int. J. Numer. Meth. Fluids, 53 (2007), pp.~1571--1583.

\bibitem{Chernov_2017_MLEnKF}
{\sc A.~Chernov, H.~Hoel, K.~Law, F.~Nobile, and R.~Tempone}, {\em Multilevel
  ensemble {Kalman} filtering for for spatio-temporal processes}, MATHICSE
  Technical Report 22.2017, EPFL, 2017,
  \url{https://www.epfl.ch/labs/mathicse/wp-content/uploads/2018/10/Report-22.2017_AC_HAH_KL_FN_RT.pdf}.

\bibitem{otpsoft}
{\sc {Computational Science Laboratory}}, {\em {ODE} test problems}, 2020,
  \url{https://github.com/ComputationalScienceLaboratory/ODE-Test-Problems}
  (accessed 2020-01-16).

\bibitem{cui2015data}
{\sc T.~Cui, Y.~M. Marzouk, and K.~E. Willcox}, {\em {Data-driven model
  reduction for the Bayesian solution of inverse problems}}, Int. J. Num. Meth.
  Eng., 102 (2015), pp.~966--990.

\bibitem{daescu2007efficiency}
{\sc D.~N. Daescu and I.~M. Navon}, {\em Efficiency of a {POD}-based reduced
  second-order adjoint model in {4D-Var} data assimilation}, International
  Journal for Numerical Methods in Fluids, 53 (2007), pp.~985--1004.

\bibitem{dihlmann2016reduced}
{\sc M.~Dihlmann and B.~Haasdonk}, {\em {A reduced basis Kalman filter for
  parametrized partial differential equations}}, ESAIM Control Optim. Calc.
  Var., 22 (2016), pp.~625--669.

\bibitem{evensen1994sequential}
{\sc G.~Evensen}, {\em {Sequential data assimilation with a nonlinear
  quasi-geostrophic model using Monte Carlo methods to forecast error
  statistics}}, Journal of Geophysical Research: Oceans, 99 (1994),
  pp.~10143--10162.

\bibitem{evensen2009data}
{\sc G.~Evensen}, {\em {Data assimilation: the ensemble Kalman filter}},
  Springer Science \& Business Media, 2009.

\bibitem{ferguson2008numerical}
{\sc J.~Ferguson}, {\em {A numerical solution for the barotropic vorticity
  equation forced by an equatorially trapped wave}}, master's thesis,
  University of Victoria, 2008.

\bibitem{foster2013finite}
{\sc E.~L. Foster, T.~Iliescu, and Z.~Wang}, {\em A finite element
  discretization of the streamfunction formulation of the stationary
  quasi-geostrophic equations of the ocean}, Comput. Methods Appl. Mech.
  Engrg., 261 (2013), pp.~105--117.

\bibitem{galbally2010non}
{\sc D.~Galbally, K.~Fidkowski, K.~Willcox, and O.~Ghattas}, {\em Non-linear
  model reduction for uncertainty quantification in large-scale inverse
  problems}, Int. J. Numer. Meth. Eng., 81 (2010), pp.~1581--1608.

\bibitem{Giles_2008_MLMC}
{\sc M.~B. Giles}, {\em Multilevel {Monte Carlo} path simulation}, Operations
  Research, 56 (2008), pp.~607--617.

\bibitem{Giles_2015_MLMC}
{\sc M.~B. Giles}, {\em Multilevel {Monte Carlo} path simulation}, Acta
  Numerica, 24 (2015).

\bibitem{greatbatch2000four}
{\sc R.~J. Greatbatch and B.~T. Nadiga}, {\em Four-gyre circulation in a
  barotropic model with double-gyre wind forcing}, J. Phys. Oceanogr., 30
  (2000), pp.~1461--1471.

\bibitem{Gregory_2017_MLETPF}
{\sc A.~Gregory and C.~Cotter}, {\em A seamless multilevel ensemble transform
  particle filter}, SIAM Journal on Scientific Computing, 39 (2017),
  pp.~A2684--A2701, \url{https://doi.org/10.1137/16M1102021},
  \url{https://doi.org/10.1137/16M1102021},
  \url{https://arxiv.org/abs/https://doi.org/10.1137/16M1102021}.

\bibitem{Reich_2016_MLETPF}
{\sc A.~Gregory, C.~Cotter, and S.~Reich}, {\em Multilevel ensemble transform
  particle filtering}, SIAM Journal on Scientific Computing, 38 (2016),
  pp.~A1317--A1338, \url{https://doi.org/10.1137/15M1038232},
  \url{https://doi.org/10.1137/15M1038232},
  \url{https://arxiv.org/abs/https://doi.org/10.1137/15M1038232}.

\bibitem{hamill2001interpretation}
{\sc T.~M. Hamill}, {\em Interpretation of rank histograms for verifying
  ensemble forecasts}, Monthly Weather Review, 129 (2001), pp.~550--560.

\bibitem{he2011use}
{\sc J.~He, P.~Sarma, and L.~J. Durlofsky}, {\em {Use of reduced-order models
  for improved data assimilation within an EnKF context}}, in SPE Reservoir
  Simulation Symposium, Society of Petroleum Engineers, 2011.

\bibitem{hesthaven2015certified}
{\sc J.~S. Hesthaven, G.~Rozza, and B.~Stamm}, {\em Certified Reduced Basis
  Methods for Parametrized Partial Differential Equations}, Springer, 2015.

\bibitem{himpe2015data}
{\sc C.~Himpe and M.~Ohlberger}, {\em Data-driven combined state and parameter
  reduction for inverse problems}, Adv. Comput. Math., 41 (2015),
  pp.~1343--1364.

\bibitem{Hoel_2016_MLEnKF}
{\sc H.~Hoel, K.~J.~H. Law, and R.~Tempone}, {\em Multilevel ensemble {Kalman}
  filtering}, SIAM Journal on Numerical Analysis, 54 (2016),
  \url{https://doi.org/10.1137/15M100955X}.

\bibitem{HLB96}
{\sc P.~Holmes, J.~L. Lumley, and G.~Berkooz}, {\em Turbulence, Coherent
  Structures, Dynamical Systems and Symmetry}, Cambridge, 1996.

\bibitem{hunt2007efficient}
{\sc B.~R. Hunt, E.~J. Kostelich, and I.~Szunyogh}, {\em Efficient data
  assimilation for spatiotemporal chaos: A local ensemble transform {K}alman
  filter}, Physica D: Nonlinear Phenomena, 230 (2007), pp.~112--126.

\bibitem{ide1997unified}
{\sc K.~Ide, P.~Courtier, M.~Ghil, and A.~C. Lorenc}, {\em Unified notation for
  data assimilation: Operational, sequential and variational)}, Journal of the
  Meteorological Society of Japan. Ser. II, 75 (1997), pp.~181--189.

\bibitem{iliescu2014variational}
{\sc T.~Iliescu and Z.~Wang}, {\em Variational multiscale proper orthogonal
  decomposition: {N}avier-{S}tokes equations}, Num. Meth. P.D.E.s, 30 (2014),
  pp.~641--663.

\bibitem{jaynes2003probability}
{\sc E.~T. Jaynes}, {\em {Probability theory: The logic of science}},
  {Cambridge university press}, 2003.

\bibitem{jespersen1974arakawa}
{\sc D.~C. Jespersen}, {\em Arakawa's method is a finite-element method},
  {Journal of Computational Physics}, 16 (1974), pp.~383--390.

\bibitem{kaercher2018reduced}
{\sc M.~Kaercher, S.~Boyaval, M.~A. Grepl, and K.~Veroy}, {\em Reduced basis
  approximation and a posteriori error bounds for {4D-Var} data assimilation},
  Optim. Eng.,  (2018), pp.~1--33.

\bibitem{Kalman_1960}
{\sc R.~Kalman}, {\em A new approach to linear filtering and prediction
  problems}, Transaction of the ASME- Journal of Basic Engineering, 82 (1960),
  pp.~35--45.

\bibitem{kalnay2003atmospheric}
{\sc E.~Kalnay}, {\em {Atmospheric modeling, data assimilation, and
  predictability}}, Cambridge Univ Pr, 2003.

\bibitem{Kikuchi_2015_ROM-EnKF}
{\sc R.~Kikuchi, T.~Misaka, and S.~Obayashi}, {\em Assessment of probability
  density function based on {POD} reduced-order model for ensemble-based data
  assimilation}, Fluid Dynamics Research, 47 (2015), p.~051403,
  \url{https://doi.org/10.1088/0169-5983/47/5/051403},
  \url{https://doi.org/10.1088%2F0169-5983%2F47%2F5%2F051403}.

\bibitem{kullback1951information}
{\sc S.~Kullback and R.~A. Leibler}, {\em On information and sufficiency}, {The
  Annals of Mathematical Statistics}, 22 (1951), pp.~79--86.

\bibitem{KV01}
{\sc K.~Kunisch and S.~Volkwein}, {\em Galerkin proper orthogonal decomposition
  methods for parabolic problems}, Numer. Math., 90 (2001), pp.~117--148.

\bibitem{law2015data}
{\sc K.~Law, A.~Stuart, and K.~Zygalakis}, {\em {Data assimilation: a
  mathematical introduction}}, vol.~62, Springer, 2015.

\bibitem{lin2014efficient}
{\sc B.~Lin and D.~McLaughlin}, {\em {Efficient characterization of uncertain
  model parameters with a reduced-order ensemble Kalman filter}}, SIAM J. Sci.
  Comput., 36 (2014), pp.~B198--B224.

\bibitem{maday2015parameterized}
{\sc Y.~Maday, A.~T. Patera, J.~D. Penn, and M.~Yano}, {\em A
  parameterized-background data-weak approach to variational data assimilation:
  formulation, analysis, and application to acoustics}, Int. J. Num. Meth.
  Engng., 102 (2015), pp.~933--965.

\bibitem{MW06}
{\sc A.~J. Majda and X.~Wang}, {\em Nonlinear dynamics and statistical theories
  for basic geophysical flows}, Cambridge University Press, Cambridge, 2006.

\bibitem{mou2020data}
{\sc C.~Mou, H.~Liu, D.~R. Wells, and T.~Iliescu}, {\em Data-driven correction
  reduced order models for the quasi-geostrophic equations: A numerical
  investigation}, Int. J. Comput. Fluid Dyn.,  (2020), pp.~1--13.

\bibitem{Sandu_2015_covarianceShrinkage}
{\sc E.~Nino-Ruiz and A.~Sandu}, {\em Ensemble {Kalman} filter implementations
  based on shrinkage covariance matrix estimation}, Ocean Dynamics, 65 (2015),
  pp.~1423--1439, \url{https://doi.org/10.1007/s10236-015-0888-9},
  \url{http://dx.doi.org/10.1007/s10236-015-0888-9}.

\bibitem{Sandu_2018_Covariance-Cholesky}
{\sc E.~Nino-Ruiz and A.~Sandu}, {\em An ensemble {Kalman} filter
  implementation based on modified {Cholesky} decomposition for inverse
  covariance matrix estimation}, SIAM Journal on Scientific Computing, 40
  (2018), pp.~A867--A886, \url{https://doi.org/10.1137/16M1097031}.

\bibitem{Sandu_2019_Covariance-parallel}
{\sc E.~Nino-Ruiz and A.~Sandu}, {\em Efficient parallel implementation of
  {DDDAS} inference using an ensemble {Kalman} filter with shrinkage covariance
  matrix estimation}, Cluster Computing, 22 (2019), pp.~2211--2221,
  \url{https://doi.org/10.1007/s10586-017-1407-1},
  \url{https://doi.org/10.1007/s10586-017-1407-1}.

\bibitem{Sandu_2017_parallel-EnKF}
{\sc E.~D. Nino-Ruiz, A.~Sandu, and X.~Deng}, {\em A parallel ensemble {Kalman}
  filter implementation based on modified {Cholesky} decomposition}, Journal on
  Computational Science, in print (2017),
  \url{https://doi.org/10.1016/j.jocs.2017.04.005}.

\bibitem{pagani2017efficient}
{\sc S.~Pagani, A.~Manzoni, and A.~Quarteroni}, {\em {Efficient state/parameter
  estimation in nonlinear unsteady PDEs by a reduced basis ensemble Kalman
  filter}}, SIAM-ASA J. Uncertain., 5 (2017), pp.~890--921.

\bibitem{peherstorfer2018survey}
{\sc B.~Peherstorfer, K.~Willcox, and M.~Gunzburger}, {\em Survey of
  multifidelity methods in uncertainty propagation, inference, and
  optimization}, Siam Review, 60 (2018), pp.~550--591.

\bibitem{petrie2008localization}
{\sc R.~Petrie}, {\em Localization in the ensemble {K}alman filter}, MSc
  Atmosphere, Ocean and Climate University of Reading,  (2008).

\bibitem{popov2019bayesian}
{\sc A.~A. Popov and A.~Sandu}, {\em A {B}ayesian approach to multivariate
  adaptive localization in ensemble-based data assimilation with time-dependent
  extensions}, Nonlinear Processes in Geophysics, 26 (2019), pp.~109--122.

\bibitem{popov2020explicit}
{\sc A.~A. Popov and A.~Sandu}, {\em An explicit probabilistic derivation of
  inflation in a scalar ensemble {K}alman filter for finite step, finite
  ensemble convergence}, 2020, \url{https://arxiv.org/abs/2003.13162}.

\bibitem{popov2020stochastic}
{\sc A.~A. Popov, A.~Sandu, E.~D. Nino-Ruiz, and G.~Evensen}, {\em A stochastic
  covariance shrinkage approach in ensemble transform {K}alman filtering},
  2020, \url{https://arxiv.org/abs/2003.00354}.

\bibitem{quarteroni2015reduced}
{\sc A.~Quarteroni, A.~Manzoni, and F.~Negri}, {\em Reduced Basis Methods for
  Partial Differential Equations: An Introduction}, vol.~92, Springer, 2015.

\bibitem{rattenbury2005almost}
{\sc N.~Rattenbury}, {\em Almost Runge-Kutta methods for stiff and non-stiff
  problems}, PhD thesis, The University of Auckland, 2005.

\bibitem{reich2015probabilistic}
{\sc S.~Reich and C.~Cotter}, {\em {Probabilistic forecasting and Bayesian data
  assimilation}}, Cambridge University Press, 2015.

\bibitem{otp}
{\sc S.~Roberts, A.~A. Popov, and A.~Sandu}, {\em {ODE} test problems: a
  {MATLAB} suite of initial value problems}, 2019,
  \url{https://arxiv.org/abs/1901.04098}.

\bibitem{rubinstein1985efficiency}
{\sc R.~Y. Rubinstein and R.~Marcus}, {\em Efficiency of multivariate control
  variates in {M}onte {C}arlo simulation}, Operations Research, 33 (1985),
  pp.~661--677.

\bibitem{san2015stabilized}
{\sc O.~San and T.~Iliescu}, {\em A stabilized proper orthogonal decomposition
  reduced-order model for large scale quasigeostrophic ocean circulation}, Adv.
  Comput. Math.,  (2015), pp.~1289--1319.

\bibitem{singler2014new}
{\sc J.~R. Singler}, {\em New {POD} error expressions, error bounds, and
  asymptotic results for reduced order models of parabolic {PDE}s}, SIAM J.
  Numer. Anal., 52 (2014), pp.~852--876.

\bibitem{sirovich1987turbulence1}
{\sc L.~Sirovich}, {\em Turbulence and the dynamics of coherent structures.
  {I.} coherent structures}, Quarterly of applied mathematics, 45 (1987),
  pp.~561--571.

\bibitem{Sandu_2015_POD-inverse}
{\sc R.~Stefanescu, A.~Sandu, and I.~Navon}, {\em {POD/DEIM} strategies for
  reduced data assimilation systems}, Journal of Computational Physics, 295
  (2015), pp.~569--595, \url{https://doi.org/10.1016/j.jcp.2015.04.030},
  \url{http://dx.doi.org/10.1016/j.jcp.2015.04.030}.

\bibitem{strazzullo2018model}
{\sc M.~Strazzullo, F.~Ballarin, R.~Mosetti, and G.~Rozza}, {\em Model
  reduction for parametrized optimal control problems in environmental marine
  sciences and engineering}, SIAM J. Sci. Comput., 40 (2018), pp.~B1055--B1079.

\bibitem{strogatz2018nonlinear}
{\sc S.~H. Strogatz}, {\em Nonlinear dynamics and chaos: with applications to
  physics, biology, chemistry, and engineering}, CRC Press, 2018.

\bibitem{Tian_2011_POD-fdvar}
{\sc X.~Tian, Z.~Xie, and Q.~Sun}, {\em A pod-based ensemble four-dimensional
  variational assimilation method}, Tellus A: Dynamic Meteorology and
  Oceanography, 63 (2011), pp.~805--816,
  \url{https://doi.org/10.1111/j.1600-0870.2011.00529.x},
  \url{https://doi.org/10.1111/j.1600-0870.2011.00529.x},
  \url{https://arxiv.org/abs/https://doi.org/10.1111/j.1600-0870.2011.00529.x}.

\bibitem{tong2015nonlinear}
{\sc X.~T. Tong, A.~J. Majda, and D.~Kelly}, {\em Nonlinear stability of the
  ensemble {K}alman filter with adaptive covariance inflation}, arXiv preprint
  arXiv:1507.08319,  (2015).

\bibitem{Heemink_2006_ROM-fdvar}
{\sc P.~T.~M. Vermeulen and A.~W. Heemink}, {\em Model-reduced variational data
  assimilation}, Monthly Weather Review, 134 (2006), pp.~2888--2899,
  \url{https://doi.org/10.1175/MWR3209.1},
  \url{https://doi.org/10.1175/MWR3209.1},
  \url{https://arxiv.org/abs/https://doi.org/10.1175/MWR3209.1}.

\bibitem{xiao2018parameterised}
{\sc D.~Xiao, J.~Du, F.~Fang, C.~C. Pain, and J.~Li}, {\em {Parameterised
  non-intrusive reduced order methods for ensemble Kalman filter data
  assimilation}}, Comput. \& Fluids, 177 (2018), pp.~69--77.

\bibitem{Yaremchuk_2009_ROM-fdvar}
{\sc M.~Yaremchuk, D.~Nechaev, and G.~Panteleev}, {\em A method of successive
  corrections of the control subspace in the reduced-order variational data
  assimilation}, Monthly Weather Review, 137 (2009), pp.~2966--2978,
  \url{https://doi.org/10.1175/2009MWR2592.1},
  \url{https://doi.org/10.1175/2009MWR2592.1},
  \url{https://arxiv.org/abs/https://doi.org/10.1175/2009MWR2592.1}.

\bibitem{zerfas2019continuous}
{\sc C.~Zerfas, L.~G. Rebholz, M.~Schneier, and T.~Iliescu}, {\em Continuous
  data assimilation reduced order models of fluid flow}, Comput. Meth. Appl.
  Mech. Eng., 357 (2019), p.~112596.

\bibitem{zubair2009efficient}
{\sc H.~B. Zubair}, {\em Efficient Multigrid Methods based on Improved Coarse
  Grid Correction Techniques.}, PhD thesis, Delft University of Technology,
  Netherlands, 2009.

\end{thebibliography}

\end{document}